\documentclass[reqno,tbtags]{amsart}

\usepackage{amsmath,amstext,amsopn,amsfonts}
\usepackage{bm}
\usepackage{array}
\usepackage{blkarray}

\usepackage{color}

\usepackage{algorithmic}
\usepackage{algorithm}
\usepackage{subfigure}
\usepackage{graphicx}

\DeclareMathOperator{\VEC}{Vec}
\DeclareMathOperator{\tr}{tr}
\DeclareMathOperator{\diag}{diag}

\DeclareMathOperator{\cov}{Cov}

\DeclareMathOperator{\E}{E}
\DeclareMathOperator{\range}{span}

\newcommand{\Krylov}{\mathcal{K}}
\newcommand{\Real}{\mathbb{R}}
\newcommand{\pmx}[1]{\begin{pmatrix}#1\end{pmatrix}}
\newcommand{\mtx}[1]{\begin{matrix}#1\end{matrix}}

\newcommand{\eps}{\varepsilon}
\newcommand{\ones}{\mathbf{1}}

\newtheorem{theorem}{Theorem}
\newtheorem{property}{Property}
\newtheorem{lemma}{Lemma}
\newtheorem{corollary}{Corollary}
\newtheorem{remark}{Remark}

\newcolumntype{H}{>{\setbox0=\hbox\bgroup}c<{\egroup}@{}}

\begin{document}
\title[PCG methods for GLM estimation]{Preconditioned Conjugate Gradient methods for the estimation of
  General Linear Models.}  
\author{Paolo Foschi}%
\begin{abstract}

  The use of the Preconditioned Conjugate Gradient (PCG) method for
  computing the Generalized Least Squares (GLS) estimator of the
  General Linear Model (GLM) is considered. The GLS estimator is
  expressed in terms of the solution of an augmented system.  That
  system is solved by means of the PCG method using an indefinite
  preconditioner.  The resulting method iterates a sequence Ordinary
  Least Squares (OLS) estimations that converges, in exact precision,
  to the GLS estimator within a finite number of steps.
  The numerical and statistical properties of the estimator computed
  at an intermediate step are analytically and numerically studied.

  This approach allows to combine direct methods, used in the OLS
  step, with those of iterative methods. This advantage is exploited
  to design PCG methods for the estimation of Constrained GLMs and of
  some structured multivariate GLMs.  The structure of the matrices
  involved are exploited as much as possible, in the OLS step.  The
  iterative method then solves for the unexploited structure.
  Numerical experiments shows that the proposed methods can achieve,
  for these structured problems, the same precision of state of the
  art direct methods, but in a fraction of the time.

\end{abstract}

\maketitle

\section{Introduction}

The general linear model (GLM) is given by
\begin{align}\label{eq:glm}
  \bm{y} &= \bm{X}\bm{\beta} + \bm{\eps},
  &
  \bm\eps &\sim (0, \bm\Sigma)
\end{align}
where %
$\bm{y} \in \Real^{m}$ is the response vector, %
$\bm{X} \in \Real^{m \times n}$ is the regressor matrix %
$\bm{\beta} \in\Real^n$ is the vector of parameters to be estimated %
and %
the disturbance term $\bm{\eps} \in \Real^{m}$ has zero mean and
variance-covariance matrix $\bm{\Sigma}$. %
Throughout the paper it will be assumed that the regressor matrix $\bm{X}$
has full-column rank.
The Ordinary Least Squares (OLS) and the Generalized Least Squares
(GLS) estimators are, respectively, defined as
\begin{align}\label{eq:OLS}
  \bm{b}_{OLS} &= (\bm{X}^T\bm{X})^{-1} \bm{X}^T\bm{y}
  \intertext{and}
  \label{eq:GLS}
  \bm{b}_{GLS} &= (\bm{X}^T \bm\Sigma^{-1} \bm{X})^{-1} \bm{X}^T \bm\Sigma^{-1} \bm{y}.
\end{align}
Both the OLS and GLS estimators are linear and unbiased.  The latter
provides the Best Linear Unbiased Estimator (BLUE) when the covariance
matrix $\bm\Sigma$ is non-singular.
This limits its applicability as singular covariance matrices may
arise in several context such as multivariate analysis, econometrics
and psychometrics
\cite{Srivastava:laa2002,Tian:jspi2009,ricos:jce97,Takada:1995,Yuan:csda2008}.

Often, computing the OLS estimator is much faster than computing the
GLS estimator. This happens, for instance, for the Seemingly Unrelated
Regressions (SUR) model, which is a GLM where the response vector, the
data matrix and the covariance matrices have, respectively, the
following structure
\begin{align*}
  \bm{y} &= \pmx{\bm{y}_1 \\ \bm{y}_2 \\ \vdots \\ \bm{y}_G },
  &
  \bm{X} &= \pmx{
    \bm{X}_1 & \bm{0} & \cdots & \bm{0} \\
    \bm{0} & \bm{X}_2 & \cdots & \bm{0} \\
    \vdots & \vdots & \ddots & \vdots \\
    \bm{0} & \bm{0} & \cdots & \bm{X}_G \\
  },
\end{align*}
and
\begin{align*}
  \bm\Sigma &= \pmx{
    \omega_{11} \bm{I}_M & \omega_{12} \bm{I}_M & \cdots & \omega_{1G}\bm{I}_M \\
    \omega_{21} \bm{I}_M & \omega_{22} \bm{I}_M & \cdots & \omega_{2G}\bm{I}_M \\
    \vdots & \vdots & & \vdots \\
    \omega_{G1} \bm{I}_M & \omega_{G2} \bm{I}_M & \cdots & \omega_{GG}\bm{I}_M 
  }. 
\end{align*}
Here, the regressor matrices $\bm{X}_i \in \Real^{M \times n_i}$,
$i=1,\ldots,G$, have full column rank, the covariance matrix
$\bm\Omega = [\omega_{ij}]_{ij} \in \Real^{G \times G}$ is symmetric
and positive semi-definite and $\bm{I}_n$ denotes the $n \times n$ identity
matrix.

Because of the block diagonal structure of $\bm{X}$, the OLS
estimation consists on collecting the OLS estimator of each block,
that is
$\bm{b}_{OLS}^T = (\bm{b}_{OLS,1}^T\; \bm{b}_{OLS,1}^T\; \cdots \; \bm{b}_{OLS,1}^T)$ 
with
$\bm{b}_{OLS,i} = (\bm{X}_i^T\bm{X}_i)^{-1}\bm{X}_i^T \bm{y}_i$. 
Clearly, the computational cost of that procedure is linear on the
number of blocks $G$.

It is not the same for GLS estimation. Although the inversion of
$\bm\Sigma$ can be efficiently obtained by inverting $\bm\Omega$,
inverting or factorising the matrix $\bm{X}^T \bm\Sigma^{-1}\bm{X}$ is
a computational expensive operation whose computational complexity is
$O\big((\sum_i n_i)^3\big)$. In that case, indeed, this matrix does
not have neither the block diagonal structure of $\bm{X}$ nor the
sparse structure of $\bm\Sigma$. Direct methods that exploit the
structure in this kind of models have been proposed and studied in
\cite{FoschiBelsleyKontoghiorghes:03,%
  FoschiKontoghiorghes:03,Foschi:csda02,%
  FoschiKontoghiorghes:03b,Foschi:laa02,%
  ricos:jce00b,ricos:aice_monog,%
  ricos:jce00a,ricos:jcsda:95}.

A similar situtation arises in the estimation of the constrained
multivariate linear model
\begin{align*}
  \bm{Y} &= \bm{X}_0 \bm{B} + \bm{U},
  &
  b_{ij}&=0, \text{ for } (i,j) \in \mathcal{C},
\end{align*}
where %
$\bm{Y},\bm{U} \in \Real^{M \times N}$ are the response and
disturbance matrices, $\bm{X}_0 \in \Real^{M \times N}$ is a fixed
data matrix, $\bm{B} \in \Real^{N \times G}$ is matrix of regression
parameters to be determined having some elements constrained to 0 and
$\mathcal{C}$ is the set the indices of the constrained elements. The
disturbances matrix $\bm{U}$ have zero mean, independent and
identically distributed (iid) rows and the covariance matrix of any row is
$\bm\Omega = [\omega_{ij}]_{ij}$. More precisely, $\E[u_{ij}] = 0$ for
all $i,j$, $\E[u_{ij}u_{pq}] = 0$ if $i \neq j$ and
$\E[u_{ij}u_{ik}] = \omega_{jk}$.

If all the constraints are relaxed then the GLS and OLS estimators are
equivalent and given by
$\bm{B}_{OLS} = (\bm{X}_0^T\bm{X}_0)^{-1}\bm{X}_0^T \bm{Y}$. The cost
of that operation $O(G N^3)$ which is linear in $G$. Instead, the
original model is equivalent to the previously considered SUR model
with the regressor block $\bm{X}_i$ obtained from $\bm{X}_0$ by deleting
the columns corresponding to constrained elements of $\bm{B}$
\cite{MagnusNeudecker:book,Srivastava:book}.
In this case, instead of changing the non-spherical distribution of
the disturbances to a spherical one, the operation that led to a
faster estimation is the relaxation of a set of constraints.

\bigskip

The aim of this work is to propose numerical algorithms, based on the
preconditioned conjugate (PCG) method, for structured linear models.
The methods here presented take advantage of the fact that changing or
relaxing some model's assumptions allows for a very fast
estimation. Here, the GLS estimator is reformulated as the solution of
an augmented system which, in turn, is solved by means of a PCG method
using an indefinite preconditioner \cite{LuksanVlcek:98}.  The
resulting method will be called PCG-Aug.
Although this method is already well known in the numerical linear
algebra community, it has not been considered in the context of
statistical estimation \cite{AltmanGondzio:99,BenziGolubLiesen:05,%
  BonettiniRuggieroTinti:07,LuksanVlcek:98,%
  RozloznikSimoncini:02}. %
The scope of the present paper is to fill this gap by deriving the
statistical properties of the resulting parameter's estimator and to
use this method to exploit the specific structure of some classes of
linear statistical models.

\bigskip

The rest of the paper is structured as follows.  Section \ref{sec:PCG}
reviews the PCG method and some of its properties. Next, in Section
\ref{sec:augPcg}, the GLS estimator is reformulated as the solution to
an augmented system. That formulation is more general than
\eqref{eq:GLS} since, under appropriate conditions, delivers a BLUE
even when the covariance matrix is singular. Then, the indefinite
preconditioner for the augmented system is reviewed and the resulting
PCG-Aug method is studied. There, in addition to some results already
discussed in \cite{AltmanGondzio:99,BenziGolubLiesen:05,%
  BonettiniRuggieroTinti:07,LuksanVlcek:98,%
  RozloznikSimoncini:02} specific issues concerning GLM estimation are
considered. In Section \ref{sec:conv} is discussed how rescaling the
covariance matrix affects the convergence of the method.  Inferential
properties of the iterates are examined both theoretically and
experimentally in Section \ref{sec:glmpcgUnbiased}.
Then, in Section \ref{sec:apps}, the PCG-Aug method is adapted to some
structured GLMs. The following models are considered: the GLM
with linear restrictions on the parameters, the restricted
multivariate GLM and the SUR model.  The performances of the proposed
methods are tested on a macro-econometric model and on Vector
AutoRegressive (VAR) models with parameter restrictions. Finally, in
the last section, conclusions and future research directions are
given.

\subsection{Notation}
The $m\times n$ matrices having all zero and all one elements are
denoted by $\bm{0}_{m \times n}$ and $\ones_{m \times n}$,
respectively. Analogously, $\bm{0}_{n}$ and $\ones_n$ denote,
respectively, the $n \times 1$ vector of all zero and all ones.  The
$n \times n$ identity matrix is denoted by $\bm{I}_n$.  Often, the
indices will be omitted if the dimension can be deduced from the
context. When dealing with multivariate GLMs, for notational
convenience, the $\VEC$ operator, direct sums and Kronecker products
of matrices will be used
\cite{ricos:aice_monog,MagnusNeudecker:book}. The $\VEC$ operator is
the operator that stacks the columns of its argument one under the
other, that is for
$\bm{A} = \pmx{\bm{a}_1 & \bm{a}_2 & \cdots & \bm{a}_n}$,
$\VEC(\bm{A}) = \pmx{\bm{a}_1^T & \bm{a}_2^T & \cdots &
  \bm{a}_n^T}{}^T$.
The Kronecker product of the matrices $\bm{A} \in \Real^{m \times n}$
and $\bm{B} \in \Real^{p \times q}$ and the direct sum of the matrices
$\bm{C}_1, \ldots, \bm{C}_G$ are, respectively, defined as
\begin{align*}
  \bm{A} \otimes \bm{B} &= \pmx{ 
    a_{11}\bm{B} & a_{12}\bm{B} & \cdots & a_{1n}\bm{B} \\
    a_{21}\bm{B} & a_{22}\bm{B} & \cdots & a_{2n}\bm{B} \\
    \vdots & \vdots & & \vdots \\
    a_{m1}\bm{B} & a_{m2}\bm{B} & \cdots & a_{mn}\bm{B} 
  }
  &&\text{and}&
  \oplus_i \bm{C}_i &= 
  \pmx{ 
    \bm{C}_1 & \bm{0} & \cdots & \bm{0} \\
    \bm{0} & \bm{C}_2 & \cdots & \bm{0} \\
    \vdots & & \ddots & \vdots \\
    \bm{0} & \bm{0} & \cdots & \bm{C}_G
  }.
\end{align*}

\section{The Preconditioned Conjugate Gradient Method}
\label{sec:PCG}

In order to fix the notation and to recall some known results, the PCG
method is reviewed. The results here reported are standard and can be
found in several monographs \cite{%
  Broyden:book,Golub:book,Greenbaum:book,
  HagemanYoung:book,LiesenStrakos:book,Saad:book,%
  VanDerVorst:book}. %
Here, the approach, terminology and notation of \cite{Broyden:book}
are followed.

The PCG method for solving the $N\times N$ symmetric linear system
$\bm{Gx} = \bm{h}$ is reported in Algorithm \ref{alg:pcg}. There,
$\bm{K} \in \Real^{N \times N}$ is an auxiliary or preconditioning
symmetric matrix, $\bm{x}_i$ is the $i$-th approximation to the
solution $\bm{x}$ and $\bm{f}_i = \bm{Gx}_i - \bm{h}$ is the
corresponding residual. Hereafter, $\bm{G}$ and $\bm{K}$ are assumed
symmetric, but not necessarily positive definite.
\begin{algorithm}[h]
  \small
  \caption{The PCG method}\label{alg:pcg}
  \begin{algorithmic}[1]
    \STATE Given $\bm{x}_1$ arbitrary
    \STATE $\bm{f}_1 =\bm{Gx}_1 -\bm{h}$, 
    $\bm{p}_1 = \bm{K} \bm{f}_1$, 
    $c_1 = \bm{f}_1^T \bm{K} \bm{f}_1$
    \FOR{$i = 1, 2, \ldots$}
    \STATE $d_i = \bm{p}_i^T \bm{G} \bm{p}_i$
    \STATE $\lambda_i = c_i / d_i$
    \STATE $\bm{x}_{i+1} = \bm{x}_i - \lambda_i \bm{p}_i $
    \STATE $\bm{f}_{i+1} = \bm{f}_i - \lambda_i \bm{Gp}_i$
    \STATE $c_{i+1} = \bm{f}_{i+1}^T \bm{Kf}_{i+1} $
    \STATE $\mu_i = c_{i+1}/c_i$
    \STATE $\bm{p}_{i+1} = \bm{Kf}_{i+1} + \mu_i \bm{p}_i$
    \ENDFOR
  \end{algorithmic}
\end{algorithm}
The following properties resume key relations among the iterates of
the PCG method under the assumption of exact precision computations.

The first property places the PCG method into the class of Kyrlov
methods and will be used in the following to characterize $\bm{p}_i$
and $\bm{f}_i$ in the context of the GLM estimation.
\begin{property}\label{t:krylov}
  Let
  \begin{align*}
    \bm{V}_i := \pmx{
      \bm{f}_1 & \bm{GK}\bm{f}_1 & (\bm{GK})^2\bm{f}_1 & \cdots & (\bm{GK})^i \bm{f}_1
    }    
  \end{align*}
  be the Krylov matrix of order $i$ generated by $\bm{GK}$ and
  $\bm{f}_1$.  The residuals and directions vectors belong,
  respectively, to the rank of $\bm{V}_i$ and of $\bm{KV}_i$, that is
  $\bm{f}_i = \bm{V}_i \bm{\gamma}$ and
  $\bm{p}_i \in \bm{KV}_i \bm\theta$, for some
  $\bm\gamma, \bm\theta \in \Real^{i+1}$.
\end{property}

The next property is an orthogonality property that the PCG's 
direction and residual vectors satisfy by construction.

\begin{property}\label{t:orth}
  The following orthogonality and $\bm{G}$-conjugacy properties hold
  \begin{align}\label{eq:orth}
    \bm{p}_j^T \bm{f}_i = 0
    \qquad\text{and}\qquad
    \bm{p}_j^T \bm{Gp}_i = 0,
    \qquad\text{for}\quad
    j<i.
  \end{align}
\end{property}
When $\bm{G}$ is positive definite, from Properties \ref{t:krylov} and
\ref{t:orth} an error minimization property follows.
\begin{property}\label{t:optim}
  Let $\bm{x}$ be a solution to $\bm{Gx} = \bm{h}$ and let $\bm{G}$ be non-negative
  definite. Then $\bm{x}_{i+1}$ minimizes the error norm
  \begin{align*}
    \varphi(\bm\xi) = \frac12 (\bm\xi-\bm{x})^T \bm{G} (\bm\xi-\bm{x}),
  \end{align*}
  on 
  $\{ \bm\xi \,|\, \bm\xi = \bm{x}_1 + \bm{KV}_i\bm\gamma,\, \bm\gamma \in \Real^i\}$
  and
  $\varphi(\bm{x}_{i+1}) \leq \varphi(\bm{x}_i)$.
\end{property}

The main consequence of Property \ref{t:orth}, is that if the method
does not breakdown ($d_i\neq 0$) or stagnate ($\bm{p}_{i+1}=\bm{p}_i$
or $\bm{f}_{i+1} = \bm{f}_{i}$) the exact solution is computed in at
most $N$ steps. To be more precise the actual number of iterations
depends on the spectrum of $\bm{GK}$:

\begin{property}\label{t:finite}
  In absence of breakdowns and stagnations, the number of steps to
  compute the exact solution is equal to the number of distinct
  eigenvalues of $\bm{GK}$.  
\end{property}

A sufficient condition for absence of breakdowns is the positive
definitiveness of both $\bm{G}$ and $\bm{K}$. The positive
definitiveness of $\bm{G}$ is problem specific and, often, it cannot
be imposed, so in order to avoid unnecessary breakdowns one would
choose a positive definite $\bm{K}$. On the other side, by property
\ref{t:finite} a computationally efficient preconditioner should
reduce the number of distinct eigenvalues of $\bm{GK}$. As pointed out
in several papers, an indefinite preconditioner similar to the one
presented in the following Section addresses that issue
\cite{AltmanGondzio:99,BenziGolubLiesen:05,BonettiniRuggieroTinti:07,LuksanVlcek:98,RozloznikSimoncini:02}.

\section{%
  The augmented system estimator and the indefinite PCG method for the
  GLM
}
\label{sec:augPcg}

\subsection{Augmented System formulation}

The GLS estimator can be computed from the solution to the augmented
system
\begin{align}\label{eq:aug1}
  \pmx{ \bm\Sigma & \bm{X} \\ \bm{X}^T & \bm{0} }
  \pmx{ \bm{w} \\ \bm{b}_{Aug} }
  =
  \pmx{ \bm{y} \\ \bm{0} },
\end{align}
where $\bm\Sigma \bm{w}$ corresponds to the residual vector of the GLM
\eqref{eq:glm}.  When $\bm\Sigma$ is positive definite \eqref{eq:aug1}
is equivalent to \eqref{eq:GLS}. However, the augmented system
formulation is more general as it does not necessarily requires a non
singular covariance matrix \cite{Kourouklis:81,Rao:book73}.
As shown in the following Lemma \ref{thm:AugBLUE}, for obtaining a
BLUE it suffices to assume that $\bm\Sigma$ is postive definite on the
null space of $\bm{X}$. More precisely,
\begin{align}\label{eq:SigmaPosXNull}
  \bm{X}^T \bm{v} &= \bm{0}, \; \bm{v} \neq \bm{0},
  &&\Rightarrow&
  \bm{v}^T\bm\Sigma \bm{v} &> 0, 
\end{align}
for any $\bm{v} \in \Real^m$.

The results presented in the following are based on the QR
decomposition of the regressor matrix $\bm{X}$, which is given by
\begin{align*}
  \bm{Q}^T \bm{X} &= \pmx{ \bm{R} \\ \bm{0} }\mtx{ n\hfill \\ m-n},
  &
  \bm{Q} &= \bordermatrix{ & n & m-n \cr &  \bm{Q}_R & \bm{Q}_N },
\end{align*}
where $\bm{Q} \in \Real^{m\times m}$ is orthogonal, that is
$\bm{Q}^T\bm{Q}=\bm{I}$ and $\bm{R} \in \Real^{n\times n}$ is
non-singular. In particular, the columns of $\bm{Q}_R$ and $\bm{Q}_N$
form orthogonal bases for the space spanned by the regressor
observations in $\bm{X}$ and its orthogonal complement, that is the rank
and the null space of $\bm{X}$.

\begin{lemma}\label{thm:AugBLUE}
  The augmented system \eqref{eq:aug1} is non-singular when
  $\bm{Q}_N^T \bm\Sigma \bm{Q}_N$ is non-singular and in that case its
  solution is given by
  \begin{subequations}
    \label{eq:augsol}
    \begin{align}
      \bm{w} &= \bm{Q}_N (\bm{Q}_N^T \bm\Sigma \bm{Q}_N)^{-1} \bm{Q}_N^T \bm{y}
      \intertext{and}
      \bm{b}_{Aug} &= \bm{R}^{-1} \bm{Q}_R^T \bm{P}_N \bm{y},
    \end{align}
  \end{subequations}
  where $\bm{P}_N = \bm{I} - \bm\Sigma \bm{Q}_N (\bm{Q}_N^T\bm\Sigma \bm{Q}_N)^{-1} \bm{Q}_N^T$.
  Moreover, $\bm{b}_{Aug}$ is a BLUE for $\bm\beta$, the vector of parameters
  of the GLM \eqref{eq:glm} \cite{Kourouklis:81}.
  
  \proof
  See Appendix \ref{sec:proofs}
\end{lemma}

\subsection{The PCG-Aug method}
\label{sec:PCG-Aug}

The PCG method presented in Section \ref{sec:PCG} is now applied to
the computation of the solution to the augmented system
\eqref{eq:aug1} 
\begin{align}\label{eq:Gaug}
  \bm{G} &=   \pmx{ \bm\Sigma & \bm{X} \\ \bm{X}^T & \bm{0}   },
  &
  \bm{x} =&  \pmx{ \bm{w} \\ \bm{z} }
  &\text{and}&&
  \bm{h} &=  \pmx{ \bm{y} \\ \bm{0} }.
\end{align}
The dimension of that system is $N = m+n$.
The iterates for $\bm{z}$ approximate the parameter estimator $\bm{b}_{Aug}$.
The auxiliary matrix $\bm{K}$ is chosen following the indefinite
preconditioner approach proposed in \cite{LuksanVlcek:98}, is used:
\begin{align}\label{eq:invprec}
  \bm{K} = \pmx{ \bm{D} & \bm{X} \\ \bm{X}^T & \bm{0} }^{-1},
\end{align}
where $\bm{D} \in \Real^{m \times m}$ is an arbitrary symmetric and
non-singular matrix, meant to approximate the dispersion matrix
$\bm\Sigma$. In the limit case of $\bm{D}= \bm\Sigma$,
$\bm{GK} = \bm{I}$ and the PCG method will compute the exact solution
in only one step.
As $\bm{X}^T\bm{D}^{-1}X$ is non singular, an explicit expression for
$\bm{K}$ is the following
\begin{align}\label{eq:prec}
  \bm{K} =
  \pmx{ \bm\Pi & \bm{X}^{\star} \\ \bm{X}^{\star T} & - (\bm{X}^T \bm{D}^{-1}\bm{X})^{-1} },
\end{align}
where %
\begin{align}\label{eq:Xstar}
  \bm{X}^{\star} &= \bm{D}^{-1} \bm{X} (\bm{X}^T \bm{D}^{-1} \bm{X})^{-1}
  &&\text{and}&
  \bm\Pi &= (\bm{I} - \bm{X}^{\star} \bm{X}^T) \bm{D}^{-1}.
\end{align}
Notice that $\bm{X}^T\bm{X}^{\star} = \bm{I}$, 
$\bm\Pi \bm{X} = \bm{0}$ and $\bm\Pi \bm{D} \bm\Pi=\bm\Pi$, 
that is
$\bm{X}^{\star}$ is a pseudo-inverse of $\bm{X}$ and $\bm\Pi$ is an oblique
projection on the null space of $\bm{X}$.

The following lemma, that can be found in \cite{LuksanVlcek:98}, shows
that, this choice for the $\bm{K}$ reduces the number of steps to at
most $m-n+1$.
\begin{lemma}
  Let $\bm{G}$ and $\bm{K}$ be defined in \eqref{eq:Gaug} and
  \eqref{eq:invprec}, respectively.  Then, $\bm{GK}$ has at least $2n$ unit
  eigenvalues.
  \proof
  It is easy to verify that
  \begin{align}\label{eq:GK}
    \bm{GK} = \bm{I} +(\bm{G}-\bm{K}^{-1})\bm{K}
    = \pmx{
      \bm{H} & (\bm\Sigma-\bm{D})\bm{X}^{\star} \\
      \bm{0} & \bm{I}_n
    },
  \end{align}
  where $\bm{H} = \bm{I}_m + (\bm\Sigma-\bm{D})\bm\Pi$.  As $\bm{GK}$
  is upper triangular, with bottom-left identity block, it has $n$
  unit eigenvalues and the remaining ones correspond to those of its
  top-left block $\bm{H}$.
  Now, because $\bm{HX}= \bm{X}$ and $\bm{X}$ has full-cloumn rank, $\bm{H}$ has at
  least $n$ unit eigenvalues. Concluding $\bm{GK}$ has at least $2n$ unit
  eigenvalues and the remaining ones are given by the non-unit
  eigenvalues of $\bm{H}$.
  \endproof
\end{lemma}

\begin{corollary}
  In exact precision and in absence of breakdowns, the PCG method with
  the indefinite preconditioner defined in \eqref{eq:invprec}, needs
  at most $m-n+1$ iterations to convergence. The convergence profile
  is determined by the spectrum of $(\bm\Sigma-\bm{D})\bm\Pi$.
\end{corollary}

This Corollary indicates a further convergence speed-up that can be
achieved by properly choosing $\bm{D}$. This choice is application
specific, as it depends on the structure of the covariance matrix
$\bm\Sigma$,

The block upper-triangular structure of $\bm{GK}$ allows to further
characterize the iterates $\bm{p}_i$ and $\bm{f}_i$. Indeed, also the
powers of $\bm{GK}$ are block upper triangular, so by Property
\ref{t:krylov} it follows that, if the first iterate
$(\bm{w}_1;\bm{z}_1)$ is chosen such that $\bm{X}^T\bm{w}_1=\bm{0}$,
then $\bm{f}_1 = (\bm{r}_1 ; \bm{0} )$, and $\bm{f}_i$ and $\bm{p}_i$ have,
respectively, the structure
\begin{align*}
  \bm{f}_i &= \pmx{ \bm{r}_i \\ \bm{0}},
  &
  \bm{p}_i & = \bm{K}\pmx{ \bm{t}_i \\ \bm{0}} 
  =: \pmx{ \bm{u}_i \\ \bm{v}_i },
\end{align*}
where $\bm{r}_i,\bm{t}_i \in \tilde{\Krylov}_i 
:= \range( \bm{r}_1, \bm{Hr}_1,\ldots, \bm{H}^i\bm{r}_1)$.
More specifically,
\begin{align}
  \label{eq:uivi} 
  \bm{u}_i &= \bm\Pi \bm{t}_i
  &\text{and}&&
  \bm{v}_i &= \bm{X}^{\star T} \bm{t}_i ,
\end{align}
that is $\bm{u}_{i}$ and $\bm{v}_i$ belong, respectively, to the null
and to the range spaces of $\bm{X}^T$. These results allow to simplify
computations in Algorithm \ref{alg:pcg}, indeed
\begin{align*}
  d_i &= \bm{u}_i^T\bm\Sigma \bm{u}_i
  &
  \text{and}&&
  c_i &= \bm{r}_i^T \bm\Pi \bm{r}_i.
\end{align*}
The requirement of having a null lower block in $\bm{f}_1$ can be
easily met by choosing a null initial guess for $\bm{w}$ or one
belonging to the null space of $\bm{X}^T$:
$\bm{X}^T\bm{w}_1 = \bm{0}$.  Moreover, convergence needs to be
verified only on the first part of the residual vector because
$\bm{X}^T\bm{w}_i = \bm{0}$.
The resulting method is resumed in Algorithm \ref{alg:glmpcg}.

\begin{algorithm}[htb]
  \small
  \caption{}
  \label{alg:glmpcg}
  \begin{algorithmic}[1]
    \STATE %
    Given $\bm{z}_1$ arbitrary 
    and $\bm{w}_1$ such that $\bm{X}^T\bm{w}_1=\bm{0}$,
    \STATE %
    $\bm{r}_1 = \bm\Sigma \bm{w}_1 + \bm{X z}_1 - \bm{y}$, 
    $c_1 = \bm{r}_1^T \bm\Pi \bm{r}_1$
    \STATE %
    $\bm{u}_1 = \bm\Pi \bm{r}_1$, 
    $\bm{v}_1 = \bm{X}^{\star T} \bm{r}_1$ 
    \FOR{$i=1,2,\ldots, m-n+1$}
    \STATE %
    $d_i = \bm{u}_i^T\bm\Sigma \bm{u}_i$
    \STATE %
    $\lambda_i = c_i / d_i$, 
    \quad
    $\bm{z}_{i+1} = \bm{z}_i - \lambda_i \bm{v}_i$,
    \quad 
    $\bm{w}_{i+1} = \bm{w}_i - \lambda_i \bm{u}_i$
    \STATE %
    $\bm{r}_{i+1} = \bm{r}_i - \lambda_i (\bm\Sigma \bm{u}_i + \bm{Xv}_i )$,
    \STATE %
    $c_{i+1} = \bm{r}_{i+1}^T \bm\Pi \bm{r}_{i+1}$
    \IF{$c_{i+1}$ is small enough}
      \STATE terminate
    \ENDIF
    \STATE %
    $\mu_i = c_{i+1}/c_i$,
    \quad  
    $\bm{v}_{i+1} = \bm{X}^{\star T} \bm{r}_{i+1} + \mu_i \bm{v}_i$,
    \quad  
    $\bm{u}_{i+1} = \bm\Pi \bm{r}_{i+1} + \mu_i \bm{u}_i$  
    \ENDFOR
  \end{algorithmic}
\end{algorithm}

Theorem 3.5 in \cite{LuksanVlcek:98} states that when both $\bm{D}$
and $\bm{Q}_N^T \bm\Sigma \bm{Q}_N$ are positive definite, the PCG-Aug
method finds the value $\bm{w}$ that solves \eqref{eq:aug1} after at
most $m-n$ iterations.
If a breakdown does not occur in the successive step, the algorithm
will retrieve the $\bm{z}$ component of the solution. 
In the experience of the author, such a breakdown is likely to arise
at that iteration.  Nonetheless, the full solution can be recovered
from $\bm{w}$. Suppose the exact $\bm{w}$ is computed at the
$i^{\star}$-th iteration, that is $\bm{w}_{i^\star} = \bm{w}$, then
$\bm\Sigma \bm{w}_{i^\star} + \bm{X}\bm{z} = \bm{y}$ and thus
\begin{align*}
  \hat{\bm{z}}_{i^\star} = \bm{X}^{\star T}( \bm{y}- \bm\Sigma \bm{w}_{i^\star} )
\end{align*}
is the solution to \eqref{eq:aug1}. Note that, the approximation
$\bm{z}_{i^\star}$ is not needed for that computation.
The complete method which takes into consideration these issues is
given in Algorithm \ref{alg:glmpcg2} and will be called PCG-Aug. The
algorithm terminates when the seminorm $\bm{r}_i^T \bm{\Pi} \bm{r}_i$
is not anymore able to decrease, when it is small enough, or when both
conditions occur.

\begin{algorithm}[htb]
  \small
  \caption{The PCG-Aug method}\label{alg:glmpcg2}
  \begin{algorithmic}[1]
    \STATE %
    Given $\bm{z}_1$ arbitrary 
    and $\bm{w}_1$ such that $\bm{X}^T\bm{w}_1=\bm{0}$,
    \STATE %
    $\bm{r}_1 = \bm\Sigma \bm{w}_1 + \bm{X z}_1 - \bm{y}$, 
    $c_1 = \bm{r}_1^T \bm\Pi \bm{r}_1$
    \STATE %
    $\bm{u}_1 = \bm\Pi \bm{r}_1$, 
    $\bm{v}_1 = \bm{X}^{\star T} \bm{r}_1$ 
    \FOR{$i=1,2,\ldots, m-n+1$}
    \STATE %
    $d_i = \bm{u}_i^T\bm\Sigma \bm{u}_i$
    \STATE %
    $\lambda_i = c_i / d_i$, 
    \quad 
    $\bm{w}_{i+1} = \bm{w}_i - \lambda_i \bm{u}_i$
    \STATE %
    $\bm{r}_{i+1} = \bm{r}_i - \lambda_i (\bm\Sigma \bm{u}_i + \bm{Xv}_i )$,
    \STATE %
    $c_{i+1} = \bm{r}_{i+1}^T \bm\Pi \bm{r}_{i+1}$
    \IF{$c_{i+1}$ is small enough}
    \STATE %
    terminate and return 
    $\hat{\bm{z}}_{i+1} = \bm{X}^{\star T}( \bm{y} - \bm\Sigma \bm{w}_{i+1})$
    \ENDIF
    \STATE $\mu_i = c_{i+1} / c_i$,
      \quad $\bm{v}_{i+1} = \bm{X}^{\star T} \bm{r}_{i+1} + \mu_i \bm{v}_i$,
      \quad $\bm{u}_{i+1} = \bm\Pi \bm{r}_{i+1} + \mu_i \bm{u}_i$  
    \ENDFOR
  \end{algorithmic}
\end{algorithm}
Another version of the same method can be obtained by considering the
following decomposition
\begin{align*}
  \bm{s}_i &= \bm{D} \bm{u}_i + \bm{X} \bm{v}_i,
  &
  \bm{u}_i &= \bm\Pi \bm{s}_i, \quad
  \bm{v}_i = \bm{X}^{*T} \bm{s}_i,
\end{align*}
since $\bm{\Pi} \bm{X} = \bm{0}$, $\bm\Pi \bm{D} \bm\Pi= \bm\Pi$ and
$\bm{X}^{*T}\bm{D} \bm\Pi=\bm{0}$. Then, the iterations for $\bm{u}_i$
and $\bm{v}_i$ in Step 12 of Algorithm \ref{alg:glmpcg2} can be
replaced by the recurrence
$\bm{s}_{i+1} = \bm{r}_{i+1} + \mu_i \bm{s}_i$, $\bm{s}_1 = \bm{r}_1$.
The resulting method is given in Algorithm \ref{alg:glmpcg3}.
\begin{algorithm}[htb]
  \small
  \caption{The PCG-Aug method (alternative version)}\label{alg:glmpcg3}
  \begin{algorithmic}[1]
    \STATE %
    Given $\bm{z}_1$ arbitrary 
    and $\bm{w}_1$ such that $\bm{X}^T\bm{w}_1=\bm{0}$,
    \STATE %
    $\bm{r}_1 = \bm\Sigma \bm{w}_1 + \bm{Xz}_1 - \bm{y}$, 
    $c_1 = \bm{r}_1^T \bm\Pi \bm{r}_1$
    \STATE $\bm{s}_1 = \bm{r}_1$
    \FOR{$i=1,2,\ldots, m-n+1$}
    \STATE $d_i = \bm{s}_i^T\bm{\Pi\Sigma\Pi} \bm{s}_i$
    \STATE $\lambda_i = c_i / d_i$, 
      \quad $\bm{w}_{i+1} = \bm{w}_i - \lambda_i \bm{\Pi s}_i$
    \STATE %
      $\bm{r}_{i+1} = \bm{r}_i 
       - \lambda_i (\bm{\Sigma \Pi} + \bm{XX}^{*T})\bm{s}_i$,
    \STATE $c_{i+1} = \bm{r}_{i+1}^T \bm{\Pi r}_{i+1}$
    \IF{$c_{i+1}$ is small enough}
    \STATE terminate and return 
    $\hat{\bm{z}}_{i+1} = \bm{X}^{\star T}( \bm{y} - \bm{\Sigma} \bm{w}_{i+1})$
    \ENDIF
    \STATE $\mu_i = c_{i+1} / c_i$,
    \quad $\bm{s}_{i+1} = \bm{r}_{i+1} + \mu_i \bm{s}_i$
    \ENDFOR
  \end{algorithmic}
\end{algorithm}
\medskip

In order to further reduce computations and to get a better
understanding of the iterates $\bm{w}_i$ computed by Algorithms
\ref{alg:glmpcg} or \ref{alg:glmpcg2}, decompose $\bm{r}_i$,
$\bm{u}_i$ and $\bm{w}_i$ on their components on the range and null
spaces of $\bm{X}$:
\begin{align}\label{eq:urwProj}
  \bm{r}_i &= \bm{Q}_N \tilde{\bm{r}}_i + \bm{Q}_R \hat{\bm{r}}_i,
  &
  \bm{u}_i &= \bm{Q}_N \tilde{\bm{u}}_i
  &\text{and}&&
  \bm{w}_i &= \bm{Q}_N \tilde{\bm{w}}_i.
\end{align}
It follows that
\begin{align*}
   c_i &= \tilde{\bm{r}}_i^T \bm{B}^{-1} \tilde{\bm{r}}_i,
  &
  d_i &= \tilde{\bm{u}}_i^T \bm{A} \tilde{\bm{u}}_i,
  &
  \tilde{w}_{i+1} &= \tilde{w}_i - \tilde{u}_i \lambda_i,
  \intertext{and}
  \tilde{u}_{i+1} &= \bm{B}^{-1} \tilde{\bm{r}}_{i+1} + \mu_i \tilde{u}_i
\end{align*}
where $\bm{A} = \bm{Q}_N^T\bm{\Sigma} \bm{Q}_N$ and
$\bm{B}=\bm{Q}_N^T \bm{D} \bm{Q}_N$. Now, the direction vectors
$\bm{v}_i$ are no longer necessary and the method for computing the
approximation $\tilde{\bm{w}}_i$ reduces to the PCG method applied to
a positive definite system with coefficient matrix $\bm{A}$ and using
$\bm{B}^{-1}$ as preconditioner (see Theorem 3.5 in
\cite{LuksanVlcek:98}). More precisely, the system solved is given
\begin{align*}
  (\bm{Q}_N^T \bm{\Sigma} \bm{Q}_N) \tilde{\bm{w}} = \bm{Q}_N^T y
\end{align*}
Property \ref{t:optim} implies that the errors norms are
non-increasing in the sense that
\begin{align}\label{eq:wOptim}
  (\bm{w}_{i+1} - \bm{w})^T \bm{\Sigma} (\bm{w}_{i+1}-\bm{w}) 
  \leq 
  (\bm{w}_i - \bm{w})^T \bm{\Sigma} (\bm{w}_i-\bm{w}).
\end{align}
Regarding the convergence, a further bound is given by
\begin{align}\label{eq:wOptim2}
  \frac{    \|\bm{w}_i - \bm{w}\|_2  }{    \|\bm{w}_1 - \bm{w}\|_2  }
  \leq
  2\sqrt{\kappa}
  \left(
    \frac{      1-\sqrt{\kappa}    }{      1 + \sqrt{\kappa}    }
  \right)^{i-1},
\end{align}
where $\kappa$ is condition number of the matrix $\bm{AB}^{-1}$, that is the
ratio between the largest and smaller eigenvalues of $\bm{AB}^{-1}$.
\medskip

From a computational point of view, it should be noted that
$\bm{X}^{\star T}\bm{r}_i$ and $\bm{\Pi} \bm{r}_i$
computed in steps 8 and 12, correspond to the GLS estimator and the
residuals of the GLM
\begin{align}\label{eq:auxglm}
  \bm{r}_i &= \bm{X} \bm{\gamma}_i + \bm{\eta},  & \bm{\eta} &\sim(\bm{0},\bm{D}).
\end{align}
That is, at each step an auxiliary GLM \eqref{eq:auxglm} need to be
estimated.
To obtain advantages from this approach, this auxiliary GLM needs to
be solved in a simple and fast manner. For instance, when a direct
method is used for that purpose, the required matrix factorizations
can be computed once at the beginning of the algorithm so that step 8
will involve only matrix multiplications and inversions of triangular
linear systems. Clearly, the cost of those factorizations depends on
the choice of $\bm{D}$.
On the other side, as previously noted, choosing $\bm{D}$ as a good
approximation to $\bm{\Sigma}$ accelerates the convergence or reduces
the number of iterations.
Then, that choice needs to balance between a good approximation to
$\bm{\Sigma}$ and a fast estimation of the GLM \eqref{eq:auxglm}.

\subsection{Scaling of $\bm\Sigma$ and convergence}
\label{sec:conv}

Eventough rescaling the covariance matrix $\bm\Sigma$ or its
approximation $\bm{D}$ has no effect on the GLS estimator, it directly
alters the spectrum of the matrix $\bm{KG}$ with consequences on the
convergence and numerical stability of the PCG-Aug method.
These effects are experimentally tested in the following
setup\footnote{ All the experiments in this work have been performed
  using Matlab R2016a on a 2.8 GHz Intel Core i7 running OS X 10.11.
}. Fixed the dimensions $m=300$ and $n=50$, $\bm{X}$ and $\bm\Sigma$
are randomly generated as follows. The first column of $\bm{X}$ is
constant and the other elements are independent samples drawn from a
normal distribution with zero mean and variance equal to $m$.  The
covariance matrix $\bm\Sigma$ has four fixed distinct eigenvalues,
$\frac12 \alpha$, $\alpha$, $\frac32 \alpha$ and $2\alpha$, each one
with multiplicity 75. The corresponding eigenvectors are randomly
generated (see the attached code for details). The auxiliary matrix
$\bm{D}$ is fixed to the identity matrix.

The convergence of the PCG-Aug method is studied for three different
values of the scaling factor: $\alpha = 1$, $\alpha=\frac14$ and
$\alpha=4$. In all the three cases the condition number of $\bm{KG}$
is not large. More precisely, that condition number is $4$ for
$\alpha=1$ and $8$ for the other two cases.  However, the convergence
and numerical performances of PCG methods are determined by the whole
spectrum of $\bm{KG}$. That spectrum is shown in Figure \ref{fig:eigA}
for the three choices of $\alpha$. For $\alpha = 1$ the spectrum of
$\bf{GK}$ has no large discontinuities and the block of unit
eigenvalues lies in the middle of the spectrum. Instead, for the other
two cases, namely $\alpha= \frac14$ and $\alpha=4$, there is a large
gap between that block of eigenvalues and the rest of the
spectrum. The presence of this gap has serious consequences on the
numerical stability of the method as shown in Figure
\ref{fig:eigB}. The non-pathological case ($\alpha=1$) shows a
convergence to a numerically precise solution much before the
theoretical bound of $m-n= 250$ steps. The other two cases exhibit the
same convergence speed but a breakdown occur before convergence is
achieved.  Since the error $\hat{z}_i - \beta$ is not known, the
convergence in terms of the norms of the residuals
$\bm{y} - \bm{X}\hat{\bm{z}}_i -\bm\Sigma\hat{\bm{z}}_i$ and
$\bm{X}^T \hat{\bm{\omega}}_i$ is also reported in Figure
\ref{fig:conv}.  The third case $\alpha=\frac14$ (not shown in that
figure) exhibit an analogous relation between errors and residuals.
Clearly, the convergence on the error can be monitored by looking at
the residuals only.

\begin{figure}[h]
  \centering
  
  \subfigure[]{\label{fig:eigA}\includegraphics*[scale=.5,trim=4 0 20 0]{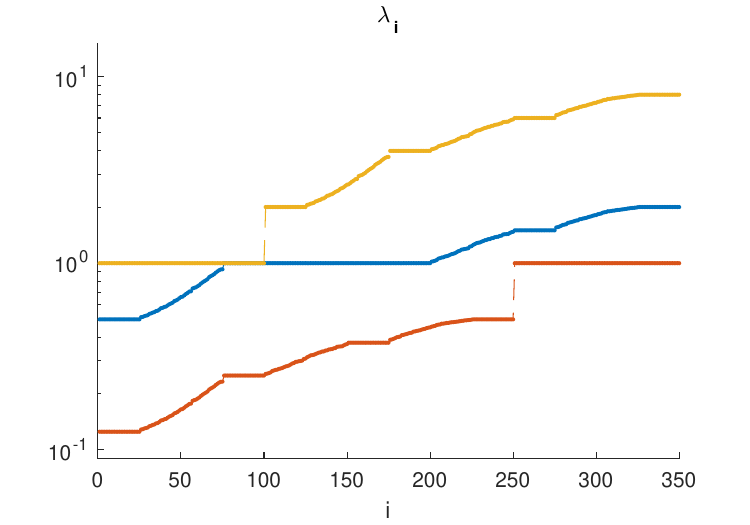}}
  \subfigure[]{\label{fig:eigB}\includegraphics*[scale=.5,trim=4 0 20 0]{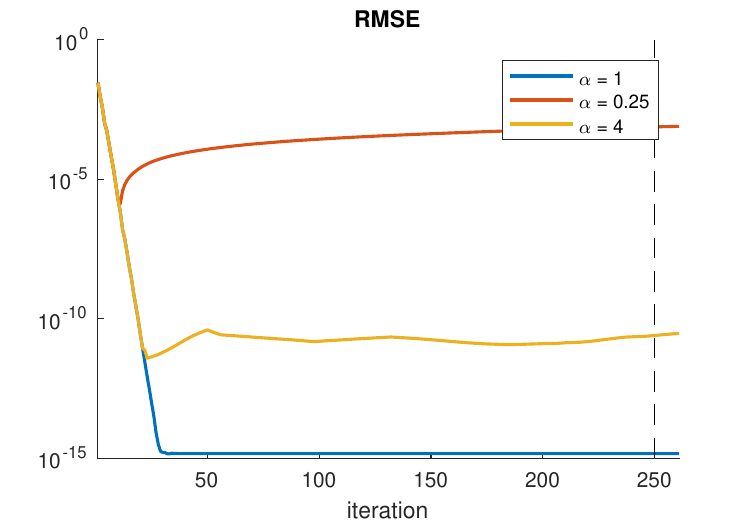}}
  
  \caption{Eigenvalue distribution of $\bf{KG}$ (left panel) and
    convergence profile (right panel). The convergence is expressed in
    terms of the root-mean-square error
    $\text{RMSE} = n^{-\frac12}\|\hat{\bm{z}}_i - \bm\beta\|_2$.}
  \label{fig:eig}
\end{figure}

\begin{figure}[htb]
  \centering

  \subfigure[]{\label{fig:convA}\includegraphics*[scale=.5,trim=4 0 20 0]{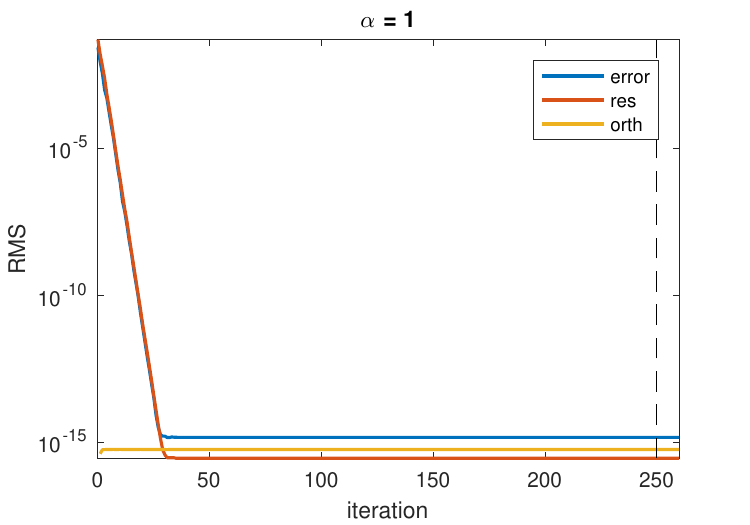}}
  \subfigure[]{\label{fig:convB}\includegraphics*[scale=.5,trim=4 0 20 0]{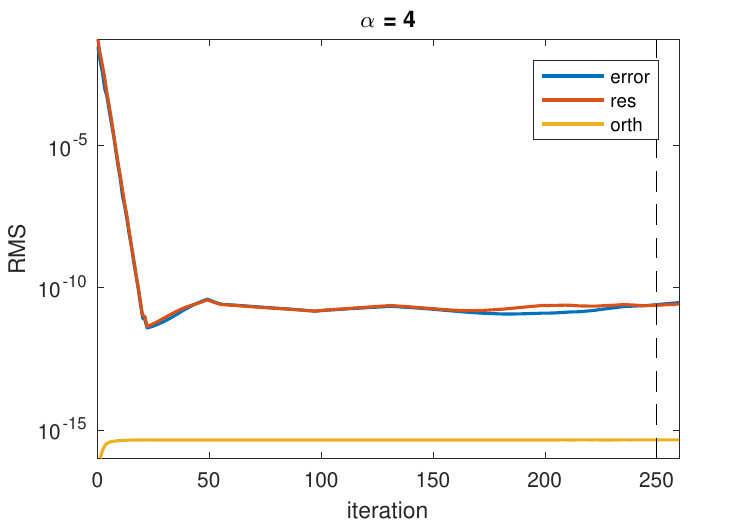}}
  
  \caption{Convergence profiles for $\alpha = 2$ (left) and
    $\alpha = 8$ (right).  The RMS (root-mean-square) of the error
    $\hat{\bm{z}}_i - \bm\beta$ and of the residuals
    $\bm{y} - \bm{X}\hat{\bm{z}}_i - \bm\Sigma\hat{\bm\omega}_i$ and
    $\bm{X}^T\hat{\bm\omega}_i$ are shown.  }
  \label{fig:conv}
\end{figure}

\subsection{Statistical properties of the PCG-Aug estimator}
\label{sec:glmpcgUnbiased}

The error of the estimator $\hat{\bm{z}}_i$ is now considered.  Recall
that
\begin{align*}
  \hat{\bm{z}}_i  = \bm{X}^{\star T} (\bm{y} - \bm{\Sigma} \bm{w}_i),
\end{align*}
so that the estimator error is given by
\begin{align*}
  \hat{\bm{z}}_i - \bm{\beta} 
  &= \bm{X}^{\star T}(\bm{\eps} - \bm{\Sigma} \bm{w}_i)
  = \bm{X}^{\star T}(\bm{\eps} - \bm{\Sigma} \bm{w}) 
  + \bm{X}^{\star T} \bm{\Sigma} ( \bm{w} - \bm{w}_i ).
\end{align*}

The following Theorem states that, for any iteration $i$, the
$\hat{\bm{z}}_i$, the PCG-Aug estimator, is an unbiased estimator for
$\bm{\beta}$ and that the transformed residuals $\bm{\omega}_i$ have zero mean.

\begin{theorem}\label{t:glmPcgUnbiased}
  If $\bm{\eps}$ is symmetrically distributed and $\bm{w}_1=\bm{0}$,
  then the iterates $\hat{\bm{z}}_i$ and $\bm{w}_i$ computed at the $i$-th
  iteration of algorithm \ref{alg:glmpcg2} have expected values
  \begin{align}\label{eQ:zUnbiased}
    \E[\hat{\bm{z}}_i] &= \bm{\beta}
    &\text{and} &&
    \E[\bm{w}_i]  &= \bm{0}.
  \end{align}
  
  \proof

  To prove \eqref{eQ:zUnbiased} it will be shown by induction that
  $\bm{u}_i, \bm{w}_i$ and $\bm{\Pi} \bm{r}_i$ are odd functions of $\bm{\eps}$ when $i>1$.
  Firstly, notice that
  $\bm{w} = \bm{Q}_N (\bm{Q}_N^T\bm{\Sigma} \bm{Q}_N)^{-1}\bm{Q}_N^T \bm{\eps}$
  is an odd function of $\bm{\eps}$.  Now, by induction, if $\bm{u}_i$ and $\bm{w}_i$
  are odd, then $d_i$ is odd. Then as
  \begin{align*}
    \bm{\Pi} \bm{r}_i = \bm{\Pi} \bm{\Sigma} (\bm{w}_i - \bm{w}) 
  \end{align*}
  is odd, $c_i = \bm{r}_i^T\bm{\Pi} \bm{r}_i$ is even and
  $\bm{w}_{i+1} = \bm{w}_i - c_i/d_i \bm{v}_i$ computed at line 6 of
  algorithm \ref{alg:glmpcg} is odd. Then, as $c_{i+1}$ is even, it
  follows that $\bm{u}_{i+1}$ computed at line 12 of the algorithm is
  odd. As $\bm{w}_1 = \bm{0}$, $\bm{\Pi} \bm{r}_1 = \bm{\Pi}\bm{\eps}$
  and $\bm{u}_1 = \bm{\Pi} \bm{r}_1$, then $\bm{w}_i$,
  $\bm{Pi}\bm{r}_i$ and $\bm{u}_i$ are odd functions of $\bm{\eps}$
  and, thus, have null expectation.
  It also follows that $\hat{\bm{z}}_i$ computed at step 10 of algorithm
  \eqref{alg:glmpcg2}, is unbiased. Indeed,
  $\hat{\bm{z}}_i = \bm{\beta} - \bm{X}^{\star T}\bm{\Sigma} \bm{w}_i$ and
  $\E[\hat{\bm{z}}_i] = \bm{\beta}$.
  \endproof
\end{theorem}

Next, the following Lemma characterizes the convergence of the errors
$\bm{z}_i-\bm{\beta}$ for $i=1,2,\ldots$.

\begin{lemma}\label{t:glmPCGconv}
  If $\bm{X}^T \bm{\Sigma} \bm{X}$ is positive definite, then
  \begin{align}\label{eq:zConv}
    (\hat{\bm{z}}_i - \bm{\beta})^T 
    (\bm{X}^{\star T}\bm{\Sigma} \bm{X}^{\star})^{-1} 
    (\hat{\bm{z}}_i - \bm{\beta})
    \leq \zeta_i,
  \end{align}
  for some decreasing sequence $\zeta_1 > \cdots > \zeta_i > \zeta_{i+1}$.
  
  \begin{proof}
  As $\bm{Q}_N^T\bm{y} = \bm{Q}_N^T \bm{\eps}$, from \eqref{eq:augsol} it follows that
  $\bm{\eps} - \bm{\Sigma} \bm{w} = \bm{P}_N \bm{\eps}$, and thus, the error reduces to
  \begin{align}\label{eq:biasz}
    \hat{\bm{z}}_i - \bm{\beta} 
    &= 
    \bm{X}^{\star T} \big( \bm{P}_N \bm{\eps} + \bm{\Sigma} ( \bm{w} - \bm{w}_i ) \big).
  \end{align}
  Now, consider the quadratic form
  $q = 
  (\hat{\bm{z}}_i - \bm{\beta})^T 
  (\bm{X}^{\star T}\bm{\Sigma} \bm{X}^{\star})^{-1} 
  (\hat{\bm{z}}_i - \bm{\beta})
  $
  which can be rewritten by means of \eqref{eq:biasz} as
  \begin{align*}
    q =
    (\bm{P}_N \bm{\eps} + \bm{\Sigma}(\bm{w} - \bm{w}_i))^T
    \bm{J} 
    (\bm{P}_N \bm{\eps} + \bm{\Sigma}(\bm{w} - \bm{w}_i) ),
  \end{align*}
  where 
  $\bm{J} 
  = \bm{X}^\star (\bm{X}^{\star T}\bm{\Sigma} \bm{X}^{\star})^{-1} \bm{X}^{\star T}
  =  \bm{X}(\bm{X}^T\bm{\Sigma} \bm{X})^{-1} \bm{X}^T$.
  By the triangle inequality, 
  \begin{align}\label{eq:expr0}
    q 
    \leq
    \bm{\eps}^T \bm{P}_N^T \bm{J} \bm{P}_N \bm{\eps}
    + 
    (\bm{w} - \bm{w}_i)^T \bm{\Sigma}\bm{J}\bm{\Sigma}(\bm{w} - \bm{w}_i).
  \end{align}
  The first term in \eqref{eq:expr0} does not depend on the iteration
  number. The second term can be bounded as follows
  \begin{align}
    \notag
    (\bm{w} - \bm{w}_i)^T \bm{\Sigma}\bm{J}\bm{\Sigma}(\bm{w} - \bm{w}_i)
    &\leq 
    (\bm{w} - \bm{w}_i)^T \bm{\Sigma}(\bm{w} - \bm{w}_i)
    \| \bm{\Sigma}^{\frac12}\bm{J}\bm{\Sigma}^{\frac12} \|
    \\
    &\leq 
    (\bm{w} - \bm{w}_i)^T \bm{\Sigma}(\bm{w} - \bm{w}_i),
    \label{eq:expr1}
  \end{align}
  where the last inequality follows from the fact that
  $\bm{\Sigma}^{\frac12}\bm{J}\bm{\Sigma}^{\frac12}$ is an idempotent
  and non-negative definite matrix and thus its maximal eigenvalue is 1.
  Finally, from \eqref{eq:expr0} and \eqref{eq:expr1} it follows
  \begin{align}\label{eq:zConvBounds}
    q \leq 
    \zeta_i :=
    \bm{\eps}^T \bm{P}_N^T \bm{J} \bm{P}_N \bm{\eps}
    +
    (\bm{w}_i-\bm{w})^T\bm{\Sigma} (\bm{w}_i-\bm{w}),
  \end{align}
  where, by \eqref{eq:wOptim2}, $\zeta_{i+1} < \zeta_i$.
  \end{proof}
\end{lemma}

Notice that, the latter result does not have a uniform nature. Indeed,
the sequence $\zeta_1 > \cdots > \zeta_{i+1}$, that bounds the
convergence profile of the estimation error
$\hat{\bm{z}}_i - \bm{\beta}$, depends ultimately on the observations
vector $y$. Then, different samples may have different convergence
profiles. However, assuming a symmetric distribution for the errors
allows to prove an uniform result on the convergence of the errors.

\begin{lemma}\label{t:covz}
  Let $\bm{\Omega}_i=\cov(\bm{w}_i-\bm{w})$, if $\bm{\eps}$ is
  symmetrically distributed and $\bm{w}_1= \bm{0}$, then
  \begin{align}\label{eq:covwDecreasing}
    \tr(\bm{\Sigma} \bm{\Omega}_{i+1} ) \leq \tr(\bm{\Sigma} \bm{\Omega}_i)
  \end{align}
  and
  \begin{align}\label{eq:covzBounds}
    \tr\!\big( (\bm{X}^{\star T}\bm{\Sigma} \bm{X}^{\star}) \cov(\hat{\bm{z}}_i)\big)
    \leq
    \tr( \bm{J} \bm{P}_N \bm{\Sigma} \bm{P}_N^T )
    +
    \tr( \bm{\Sigma} \bm{\Omega}_i  ),
  \end{align}
  where $\bm{J} = \bm{X}(\bm{X}^T\bm{\Sigma} \bm{X})^{-1}\bm{X}^T$.
  \begin{proof}
    When $\bm{\eps}$ is symmetrically distributed and
    $\bm{w}_1=\bm{0}$, Theorem \ref{t:glmPcgUnbiased} implies that
    $\cov(\bm{w}_i-\bm{w}) =
    \E[(\bm{w}_i-\bm{w})(\bm{w}_i-\bm{w})^T]$.
    The first result follows by taking the appropriate expectation
    from \eqref{eq:wOptim}.  The second result can be obtained by
    taking the expectation of \eqref{eq:zConvBounds}.
  \end{proof}
\end{lemma}


\medskip

A Monte Carlo experiment have been designed to tests these results.
In this numerical experiment the performances of the PCG-Aug method and
of the PCG-NE method are compared. Here, PCG-NE refers to a classical
conjugate gradient method applied to the normal equations
$(\bm{X}^T\bm\Sigma^{-1}\bm{X}) \bm\beta = \bm{X}^T\bm\Sigma^{-1}
\bm{y}$.
The setup is the following: $m=80$, $n=20$ and the number of MC
replications is 1000. The preconditioner is $\bm{K}$ in
\eqref{eq:invprec} with $\bm{D} = \bm{I}$. The matrices $\bm{X}$ and
$\bm\Sigma$ are randomly generated, but kept fixed for the whole
experiment. The regressor matrix is generated as in Section
\ref{sec:conv} and $\bm\Sigma$ has four distinct eigenvalues, 0.01,
0.1, 10 and 50, each one with multiplicity 50. The resulting
preconditioned coefficient matrix $\bm{KG}$ has condition number equal
to 4.95e+03 and its spectrum is shown in Figure \ref{fig:eig2}.
\begin{figure}[h]
  \centering

  \includegraphics*[scale=.5]{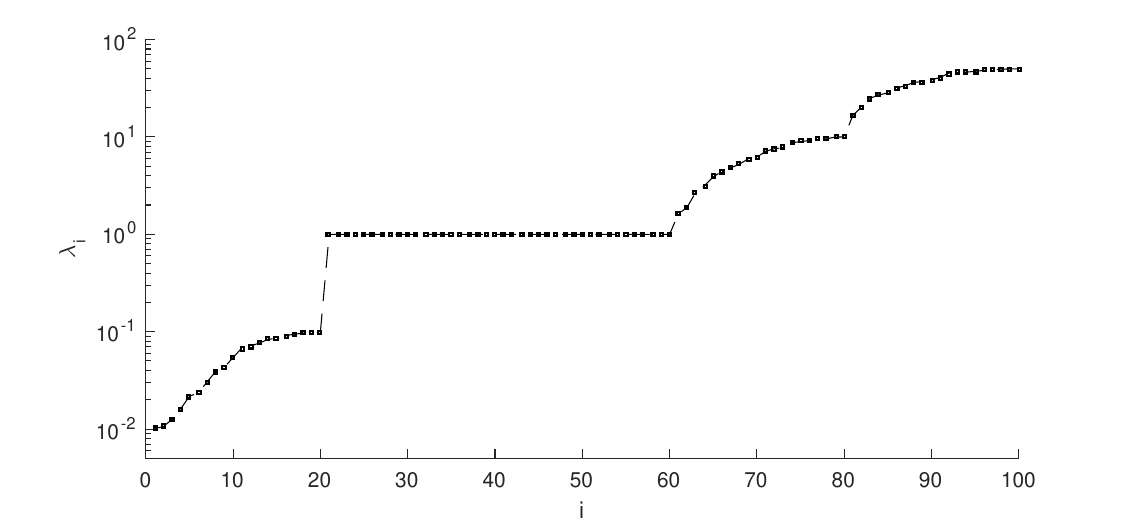}

  \caption{\small Spectrum of $\bm{KG}$. }
  \label{fig:eig2}
\end{figure}

Accordingly to \eqref{eq:glm}, in each MC replication
the observation vector $\bm{y}$ is drawn from multivariate normal
distribution with covariance matrix $\bm{\Sigma}$ and mean
$\bm{X}\bm\beta$. Figure \ref{fig:conv1} and \ref{fig:conv3} report,
respectively, elementwise and uniform results.
More precisely, Figure \ref{fig:conv1} shows results on the MC
distribution of $\beta_2$ for the PCG-Aug method (Figures
\ref{fig:conv1a} and \ref{fig:conv1c}) and for the PCG-NE method.
Figure \ref{fig:conv1a} clearly confirm part of Theorem
\ref{t:glmPcgUnbiased}. Figure \ref{fig:conv1c} and \ref{fig:conv1d}
show the superior performances of PCG-Aug both for the 1\% and 5\%
tails and for the average case. Moreover, differently than PCG-NE
which start with a very low-variance estimator and then update it
monotonically reducing the bias and increasing the variance, the
PCG-Aug keeps an unbiased estimator throughout the iterations, but
does not show a monotone behaviour.
These conclusions are confirmed looking at Figure \ref{fig:conv3}
which shows the MC bands and average for RMSE of the estimators for
$\bm\beta$.

\begin{figure}[h]
  \centering

  \subfigure[]{\label{fig:conv1a}\includegraphics*[scale=.5,trim=4 0 20 0]{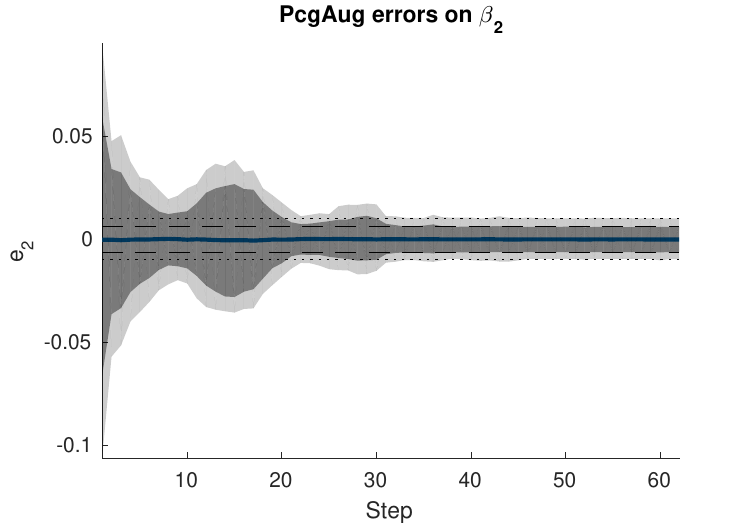}}
  \subfigure[]{\label{fig:conv1b}\includegraphics*[scale=.5,trim=4 0 20 0]{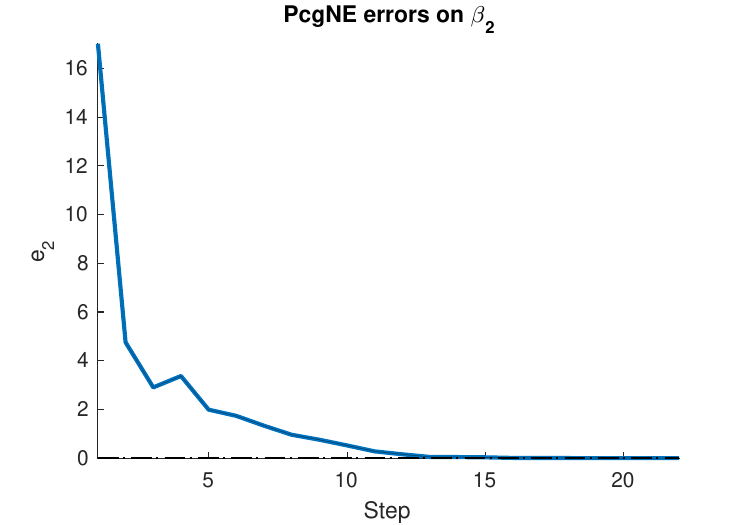}}
  
  \subfigure[]{\label{fig:conv1c}\includegraphics*[scale=.5,trim=4 0 20 0]{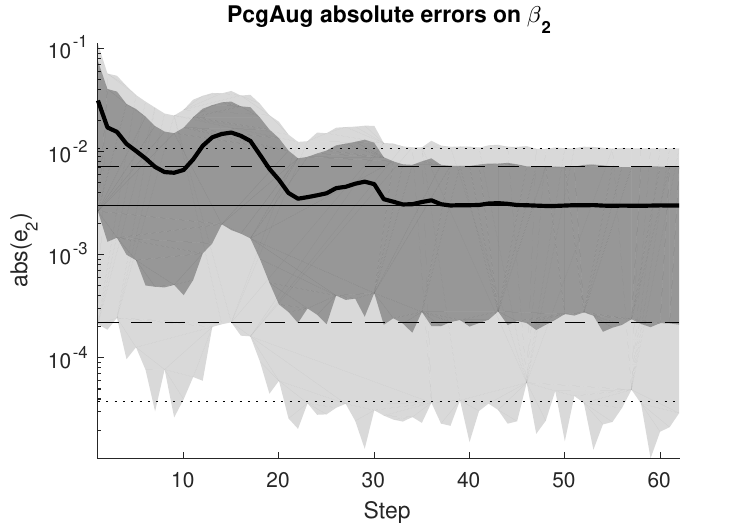}}
  \subfigure[]{\label{fig:conv1d}\includegraphics*[scale=.5,trim=4 0 20 0]{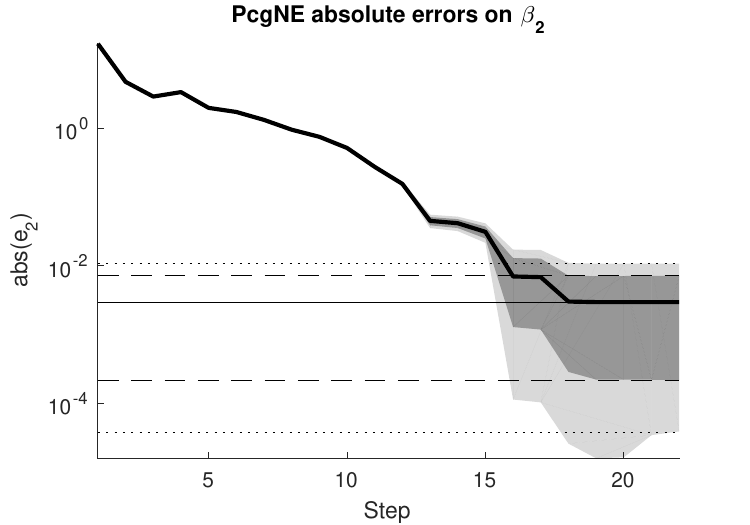}}

  \caption{\small Montecarlo average and 95\% and 99\% bands for the
    error on the second element of the PCG-Aug and PCG-NE estimators
    for $\bm\beta$. 
    Analogous statistics for the GLS estimator are shown as horizontal
    continous, dashed and dotted lines. }
  \label{fig:conv1}
\end{figure}

\begin{figure}[h]
  \centering

  \includegraphics*[scale=.5,trim=4 0 20 0]{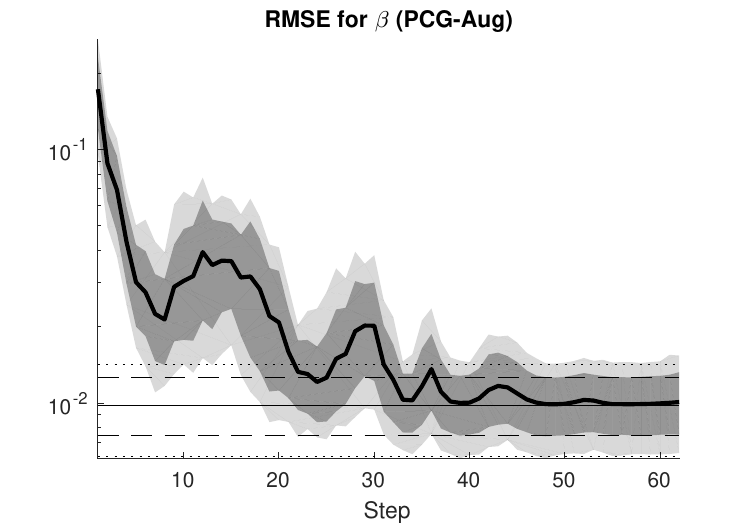}
  \hfill
  \includegraphics*[scale=.5,trim=4 0 20 0]{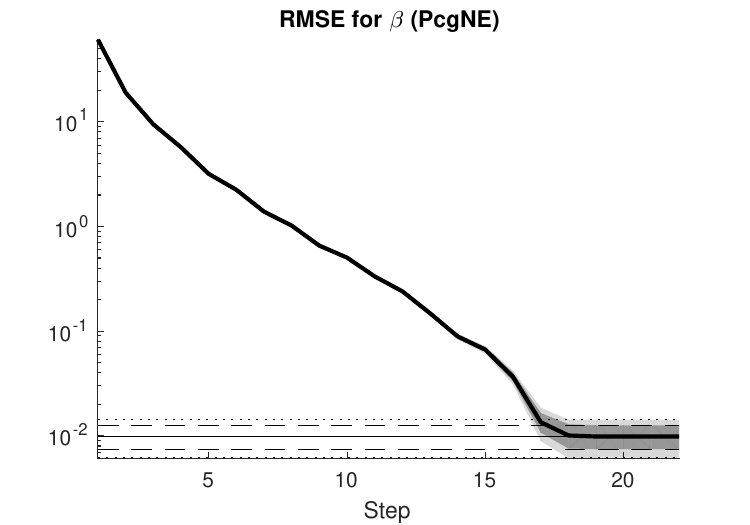}
  
  \caption{\small %
    Montecarlo average and 95\% and 99\% bands for the RMSE error
    $n^{-\frac12} \| \hat{\bm{z}}_i - \bm\beta
    \|_{(\bm{X}^T\bm\Sigma^{-1}\bm{X})}$
    of the PCG-Aug and PCG-NE estimators. }
  \label{fig:conv3}
\end{figure}

\section{Applications}
\label{sec:apps}

In this section the PCG-Aug method is adapted to the GLS estimation of
GLMs with specific structures.

\subsection{GLMs with linear restrictions on the parameters}
\label{sec:RGLM}

Consider a GLM where a set of $k$ linear restrictions are imposed on
the parameters $\bm{\beta}$:
\begin{align}\label{eq:RGLM}
  \bm{\zeta} &= \bm{Z} \bm{\beta} + \bm{\eps},
  &
  \bm{C} \bm{\beta} &= \bm{\gamma},
  &
  \bm{\eps} &\sim (\bm{0}, \bm{\Omega}),
\end{align}
where $\bm{Z} \in \Real^{m \times n}$,
$\bm{C} \in \Real^{k \times n}$ and 
$\bm{\gamma} \in \Real^{k}$ are fixed.

Often, these constraints consists on fixing some elements of
$\bm{\beta}$ to be null. In that case $\bm{C}$ is a selection matrix,
that is a matrix whose row are a subset of the rows of the identity
matrix and the rhs is null, $\bm{\gamma}=\bm{0}$. It turns out that
$\bm{C}^T$ is semi-orthogonal, that is $\bm{CC}^T=\bm{I}_k$ and applying
$\bm{C}$ is equivalent to selecting the constrained
elements. Analogously, the application of the diagonal matrices
$\bm{C}^T\bm{C}$ or $\bm{I}-\bm{C}^T\bm{C}$ has the effect of
annihilating the restricted or the unrestricted elements,
respectively.

The restricted GLM \eqref{eq:RGLM} can be seen as a GLM with $m+k$
observations. The additional $k$ observation corresponds to the
constraints and have a disturbance term with zero variance. More
precisely, the restricted model \eqref{eq:RGLM} is equivalent to the
GLM \eqref{eq:glm} with
\begin{align}\label{eq:RGLM1setup}
  \bm{\Sigma} &= \pmx{\bm{\Omega} & \bm{0} \\ \bm{0} & \bm{0}},
  &
  \bm{X} &= \pmx{ \bm{Z} \\ \bm{C} },
  &
  \bm{y} &= \pmx{ \bm{\zeta} \\ \bm{\gamma} }.
\end{align}
The GLS estimator of that GLM can be computed using Algorithms
\ref{alg:glmpcg2} and \ref{alg:glmpcg3} to solve the augmented system 
with the coefficient matrix and the RHS vector, respectively, given by
\begin{align}\label{eq:GRGLM}
  \bm{G} &= 
  \begin{blockarray}{cccl}
    m & k & n & \\
    \begin{block}{(ccc)l}
      \bm{\Omega} & \bm{0}   & \bm{Z} & m \\
      \bm{0}      & \bm{0}   & \bm{C} & k \\
      \bm{Z}^T    & \bm{C}^T & \bm{0} & n \\
    \end{block}
  \end{blockarray}
  &\text{and}&&
  \bm{h} &= \pmx{ \bm{\zeta} \\ \bm{\gamma} \\ \bm{0}}.
\end{align}
A convenient choice for the top-right block of the preconditioner matrix
in \eqref{eq:invprec} is 
\begin{align}
  \label{eq:3}
  \bm{D} = \pmx{ \bm{D}_Z & \bm{0} \\ \bm{0} & \bm{D}_C },
\end{align}
where $\bm{D}_Z \in \Real^{m \times m}$ and $\bm{D}_C \in \Real^{k \times k}$
are arbitrary symmetric and positive definite auxiliary matrices.  In
exact precision and absence of breakdowns, the maximum number of steps
required by the PCG method is $m+k-n+1$.

To use Algorithms \ref{alg:glmpcg} and \ref{alg:glmpcg2} it is
necessary to apply $\bm{X}^\star$ and $\bm{\Pi}$ to a vector. From \eqref{eq:Xstar}
and \eqref{eq:RGLM1setup} it follows that those matrices are given by
\begin{subequations}\label{eq:rglm1m}
\begin{align}
  \bm{X}^\star &= 
  \pmx{\bm{D}_Z^{-1}\bm{Z} \\ \bm{D}_C^{-1}\bm{C}} 
  (\bm{Z}^T\bm{D}_Z^{-1}\bm{Z} + \bm{C}^T\bm{D}_C^{-1}\bm{C})^{-1}
\end{align}
and
\begin{align}
  \label{eq:rglm1mB}
  \bm{\Pi} &=
  \pmx{\bm{D}_Z^{-1} & \bm{0} \\ \bm{0} & \bm{D}_C^{-1}} 
  - 
  \pmx{\bm{D}_Z^{-1}\bm{Z} \\ \bm{D}_C^{-1} \bm{C}}
  (\bm{Z}^T\bm{D}_Z^{-1}\bm{Z} + \bm{C}^T \bm{D}_C^{-1} \bm{C})^{-1}
  \pmx{\bm{Z}^T\bm{D}_Z^{-1} & \bm{C}^T \bm{D}_C^{-1}}.
\end{align}
\end{subequations}
\medskip

Now, in order to provide some insight on the behaviour of this
iterative estimation method, let consider the case where $\bm{C}$ is a
selection matrix and both $\bm{D}_Z$ and $\bm{D}_C$ are identity
matrices. In that case, 
\begin{align*}
  \bm{X}^\star &= 
  \pmx{\bm{Z} \\ \bm{C}} (\bm{Z}^T\bm{Z} + \bm{C}^T\bm{C})^{-1},
\end{align*}
and applying $\bm{K}$ to a vector of residuals
$\bm{f} = (\bm{r}^T \; \bm{0}^T)^T$ corresponds to a shinkage
regression of $\bm{r}$ against $\bm{Z}$:
\begin{align*}
  \bm{X}^{\star T}\bm{f} = (\bm{Z}^T\bm{Z} +
  \bm{C}^T\bm{C})^{-1}\bm{Z}^T\bm{r}.  
\end{align*}
Indeed, recall that $\bm{C}^T\bm{C}$ is a diagonal matrix with unit
elements in correspondence of the restrictions and zero elsewhere.  In
some sense, at each step of PCG method, the application of the
preconditioner shrinks the current estimator toward the manifold
defined by the constraints. Other approaches usually project the
estimator into that subspace.

\subsection{Multivariate linear models with parameters restrictions}
\label{sec:MVRGLM1}

The multivariate GLM is specified as follows,
\begin{align}\label{eq:MvGLM}
  \bm{Y} &= \bm{Z}_0 \bm{B} + \bm{U},
\end{align}
where $\bm{Y}, \bm{U} \in \Real^{M \times G}$ are, respectively, the
response and disturbance matrices, $\bm{Z}_0 \in \Real^{M \times N}$
is a full column rank regressor matrix and
$\bm{B} \in \Real^{N \times G}$ is the matrix of parameters. The rows
of $\bm{U}$ are iid with zero mean and covariance matrix
$\bm{\Omega}_0 \in \Real^{G \times G}$, that is
$\VEC(\bm{U}) \sim (\bm{0},\bm{\Omega}_0 \otimes \bm{I}_M)$.

For the multivariate model \eqref{eq:MvGLM}, OLS and GLS estimations
give the same estimator:
$\hat{\bm{B}} = (\bm{Z}_0^T\bm{Z}_0)^{-1}\bm{Z}_0^T \bm{Y}$.  That
equivalence is broken when linear restrictions are imposed on the
elements of $\bm{B}$ \cite{MagnusNeudecker:book}. Often, those can
simply be exclusion restriction of the kind $b_{ij} = 0$ for some set
of couples $(i,j)\in S$.
The restricted multivariate model can be written in the form of
\eqref{eq:RGLM} by setting $\bm{Z} = \bm{I}_G \otimes \bm{Z}_0$,
$\bm{\beta} = \VEC(\bm{B})$, $\bm{\eps} = \VEC(\bm{U})$,
$\bm{\Omega} = \bm{\Omega}_0 \otimes \bm{I}_M$ and collecting all the
restriction coefficients on the matrix $\bm{C}$.  The total number of
observations and regressors are $m = GM$ and $n=GN$. As above, $k$
will denote the total number of restrictions.

\subsubsection{Model reduction}
Since $\bm{Z}_0$ has been assumed with full column rank, the model can
be reduced by means of a preliminar transformation. To this end
consider the QRD $\bm{Z}_0 = \bm{Q}_0 \bm{R}_0$, where
$\bm{R}_0 \in \Real^{N \times N}$ is triangular and non-singular and
$\bm{Q}_0 \in \Real^{M \times N}$ semi-orthogonal, that is
$\bm{Q}_0^T\bm{Q}_0 = \bm{I}_N$. Then, premultiplying \eqref{eq:MvGLM}
by $\bm{Q}_0^T$ it gives
\begin{align*}
  \bm{Y}_0 &= \bm{R}_0 \bm{B} + \bm{U}_0,
  &
  \VEC(\bm{U}_0) &\sim (\bm{0}, \bm{\Omega}_0 \otimes \bm{I}_N),
\end{align*}
where $\bm{Y}_0 = \bm{Q}_0^T \bm{Y}$ and $\bm{U}_0 = \bm{Q}_0^T \bm{U}$.
The reduced restricted model can then be written in the form
\eqref{eq:RGLM} where
\begin{align*}
  \bm{Z} &= \bm{I}_G \otimes \bm{R}_0,
  &
  \bm{\beta} &= \VEC(\bm{B}),
  &
  \bm{\eps} &= \VEC(\bm{U}_0),
  &
  \bm{\Omega} &= \bm{\Omega}_0 \otimes \bm{I}_N.
\end{align*}
The total number of equation of the model is reduced from $m=GM$ to $m=GN$.

\subsubsection{Adaption of the PCG-Aug method}
Using the approach given in Section \ref{sec:RGLM} leads to a PCG-Aug
method that requires $m+k-n+1 = GN+k-GN+1 =k+1$ steps.  However, each
step requires the computation of $\bm{X}^\star$ and $\bm{\Pi}$. A task
that can be computationally expensive because it requires the
inversion of the matrix
\begin{align*}
  \pmx{\bm{Z} \\ \bm{C}}^T \pmx{\bm{Z} \\ \bm{C}}
  = \bm{I} \otimes \bm{R}_0^T\bm{R}_0 + \bm{C}^T\bm{C}.
\end{align*}
In absence of cross equation restrictions, that is when
$\bm{C} = \oplus_i \bm{C}_i$, $\bm{C}_i \in \Real^{k_i \times N}$,
that computation can be efficiently performed.  In that case, indeed,
that that $GN \times GN$ matrix is block diagonal and computing its
inverse reduces to the inversion of the $N \times N$ matrices
$\bm{R}_0^T \bm{R}_0 + \bm{C}_i^T\bm{C}_i$, $i=1,\ldots,G$.  This task
can be computed by means of the QRD of the set of matrices
\begin{align*}
  \pmx{\bm{R}_0 \\ \bm{C}_i}, && i &= 1,\ldots, G.
\end{align*}
More precisely, $\bm{X}^\star$ is given by
\begin{align*}
  \bm{X}^\star = 
  \pmx{
    \oplus_i \bm{R}_0 (\bm{R}_0^T\bm{R}_0+\bm{C}_i^T\bm{C}_i)^{-1} \\ 
    \oplus_i \bm{C}_i (\bm{R}_0^T\bm{R}_0+\bm{C}_i^T\bm{C}_i)^{-1} 
  },
\end{align*}
which, after a permutation of the rows, can be written as the block
diagonal matrix
\begin{align*}
  \tilde{\bm{X}}^\star = 
  \oplus_i 
  \pmx{\bm{R}_0 \\ \bm{C}_i}
  (\bm{R}_0^T \bm{R}_0+\bm{C}_i^T\bm{C}_i)^{-1}.
\end{align*}
Now, for each block of that matrix consider the Updating QRD
\begin{align}\label{eq:MvRGLMQRi}
  \pmx{\bm{R}_0 \\ \bm{C}_i} &= \tilde{\bm{Q}}_i \bm{R}_i, 
\end{align}
where 
$\bm{R}_i \in \Real^{N \times N}$,
$\tilde{\bm{Q}}_i \in \Real^{(N+k_i) \times N}$
and
$\tilde{\bm{Q}}_i^T \tilde{\bm{Q}}_i = \bm{I}_N$ 
\cite{Golub:book,ricos:aice_monog,ricos:simax00,ricos:rank-k}.
It turns out that each block of $\tilde{\bm{X}}^\star$ can be written as
\begin{align*}
  \pmx{\bm{R}_0 \\ \bm{C}_i}(\bm{R}_0^T \bm{R}_0+\bm{C}_i^T\bm{C}_i)^{-1}
  =
  \tilde{\bm{Q}}_i \bm{R}_i^{-T}.
\end{align*}
Analogously, $\bm{\Pi}$ in \eqref{eq:rglm1mB} is equal, modulo a
permutation, to the block diagonal matrix
\begin{align*}
  \bm{I} - \oplus_i 
  \pmx{\bm{R}_0 \\ \bm{C}_i}
  (\bm{R}_0^T \bm{R}_0+\bm{C}_i^T\bm{C}_i)^{-1} 
  \pmx{\bm{R}_0^T & \bm{C}_i^T}
  =
  \oplus_i (\bm{I} - \tilde{\bm{Q}}_i \tilde{\bm{Q}}_i^T).
\end{align*}
This set of QRDs need to be computed only once. Then, each step of the
PCG-Aug method requires the application of $\bm{R}_i^{-T}$, $\bm{Q}_i$ and
$\bm{Q}_i^T$ to some vector (for $i=1,\ldots,G$).

A summary of the computational cost of the main steps of the resulting
algorithm is reported in Table \ref{tab:MvGLM1Compl}, where the
following majorization have been used 
\begin{align*}
  \sum_{i=1}^G (N+k_i)^2
  &\leq
  G\Big( (N + \bar{k})^2
  +
  (k_{\max} - k_{\min})^2 
  \Big).
\end{align*}
\begin{table}[h]
  \small
  \centering
  \begin{tabular}{ll@{\,$\times$\,}l}
    \hline \\[-2.3ex]
    \multicolumn{1}{l}{Task} 
     & \multicolumn{1}{l}{Compl. complexity} 
     & Nr. of Iterations
    \\[.3ex]
    \hline \\[-2.3ex]
    QRD of $\bm{Z}_0$ & $M^2N$ & $1$
    \\[.3ex]
    QRDs of $\pmx{\bm{R}_0 \\ \bm{C}_i}$, $\forall i$,
      & $\sum_{i=1}^G (N + k_i)^2N $ & $1$ 
    \\[.3ex]
    Apply $\pmx{\bm{R}_0 \\ \bm{C}_i }$ to some vector, $\forall i$ 
      & $GN(N+\bar{k})$ & $G \bar{k}$
    \\[.3ex]
    Apply $\pmx{\bm{R}_0 \\ \bm{C}_i}^{*T}$ to some vector $\forall i$ 
      & $GN(N+\bar{k})$ & $G \bar{k}$
    \\[.3ex]
    Apply $\bm\Pi$  
      & $GN(N+\bar{k})$ & $G \bar{k}$
    \\[.3ex]
    \hline \\[-2.3ex]
    Total complexity
      & \multicolumn{2}{l}{
        $
        M^2N + \sum_{i=1}^G (N + k_i)^2N + 
        G^2 N^2 \bar{k} + G^2 N \bar{k}^2
        $        
        }
    \\[.3ex]
      & \multicolumn{2}{l}{
        $M^2N + GN(N + \bar{k})^2$        
        }
    \\
    \hline
  \end{tabular}
  \medskip

  \caption{%
    Computational cost of the method presented in Section 
    \ref{sec:MVRGLM1} for estimating the RGLM \eqref{eq:MvGLM}. 
  }
  \label{tab:MvGLM1Compl}
\end{table}

\subsection{Seemingly unrelated regressions model}
\label{sec:SUR}

The Seemingly unrelated regressions (SUR) model is a GLM where the
response vector and the regressor and covariance matrices hava the
following structure
\begin{align}\label{eq:SURdefs}
  \bm{y} &= \VEC(\bm{Y}),
  &
  \bm{X} &= \oplus_{i=1}^G \bm{X}_i
  &\text{and}&&
  \bm{\Sigma} &= \bm{\Sigma}_0 \otimes \bm{I}_M,
\end{align}
with $\bm{Y} \in \Real^{M \times G}$, 
$\bm{X}_i \in \Real^{M \times N_i}$,
$\bm{\Sigma}_0 \in \Real^{M \times M}$ 
\cite{FoschiBelsleyKontoghiorghes:03,MagnusNeudecker:book,Srivastava:book}.
Estimating the SUR model by a straightforward implementation of the
GLS estimator is expensive for large models, that is when both $G$ and
$M$ are big. Indeed, given the complementarity in the structure of
$\bm{X}$ and $\bm\Sigma$, computing the cross products in
\eqref{eq:GLS} leads to full matrices.  Specific factorization methods
have been designed in the past to exploit the sparsity structure of
these models
\cite{FoschiBelsleyKontoghiorghes:03,Foschi:csda02,FoschiKontoghiorghes:03,%
  FoschiKontoghiorghes:03b,ricos:jcsda:95,ricos:aice_monog}.  The
PCG-Aug method here proposed provides a valid alternative approach.

In order to estimate the SUR model by means of the PCG-Aug method,
consider the simplest choice for the preconditioner $\bm{D} = \bm{I}_{MG}$.  In
that case the preconditioner matrix $\bm{K}$ in \eqref{eq:prec} can be
explicitly computed by $\bm{X}^\star = \oplus_i \bm{X}_i^\star$ and
$\bm{\Pi} = \oplus_i \bm{\Pi}_i$ with
\begin{align*}
  \bm{X}_i^\star &= \bm{X}_i (\bm{X}_i^T\bm{X}_i)^{-1},
  &\text{and}&&
  \bm{\Pi}_i &= \bm{I}_M - \bm{X}_i (\bm{X}_i^T\bm{X}_i)^{-1}\bm{X}_i^T.
\end{align*}
Note that both $\bm{X}$ and $\bm{\Pi}$ are block diagonal matrices.
An alternative, and numerically more stable, approach for applying $\bm{K}$
derives from QR decompositions of each regressor $\bm{X}_i$:
$\bm{X}_i = \bm{Q}_i\bm{R}_i$, $i = 1,\ldots,G$, where
$\bm{R}_i \in \Real^{N_i \times N_i}$ is upper triangular, 
and $\bm{Q}_i \in \Real^{M \times N_i}$ orthogonal, that is
$\bm{Q}_i^T \bm{Q}_i = \bm{I}_{N_i}$. Then,
\begin{align*}
  \bm{X}_i^\star &= \bm{Q}_i \bm{R}_i^{-T}
  &\text{and}&&
  \bm{\Pi}_i &= \bm{I}_M - \bm{Q}_i \bm{Q}_i^{T}.
\end{align*}
Furthermore, setting $\bm{\xi} =\VEC(\{ \bm{\xi}_i \}_{i=1}^G)$,
the computation of $\bm{\Pi} \bm{\xi}$ and $\bm{X}^{\star T}\bm{\xi}$ required in
steps 8 and 12 of Algorithm \ref{alg:glmpcg2} reduces to 
\begin{align*}
  \bm{\Pi} \bm{\xi}
  &=  (\oplus_i \bm{\Pi}_i) \VEC(\{ \bm{\xi}_i \}_{i=1}^G)
  = \VEC(\{ \bm{\xi}_i - \bm{Q}_i\bm{Q}_i^T\bm\xi_i \}_{i=1}^G)
  \intertext{and}
  \bm{X}^{\star T} \bm{\xi}
  &= (\oplus_i \bm{X}_i^{*})^T
  = \VEC(\{ \bm{R}_i^{-1} \bm{Q}_i^T \bm{\xi}_i\}_{i=1}^G).
\end{align*}
Note that, the $i$-th block of $\bm{X}^{*T} \bm\xi$ is the OLS estimator of a
model with regressor $\bm{X}_i$ and response $\bm{\xi}_i$. Moreover, the $i$-th
block of $\bm{\Pi} \bm{\xi}$ is given by the corresponding residual vector.

In step 2 and 7 of Algorithm \ref{alg:glmpcg2}, the product $\bm{Xv}$
is a set matrix-vector products involving $\bm{X}_i, i=1,\ldots,G$,
and the product
$\bm{\Sigma} \bm{u} = (\bar{\bm\Sigma} \otimes \bm{I}_M)\bm{u}$ which
appears in steps 2, 5 and 7, reduces to the matrix product
$\bm{U}\bar{\bm\Sigma}$, where $\bm{U}\in\Real^{M \times G}$ is such that
$\bm{u} = \VEC(\bm{U})$.

Resuming, each iteration of the PCG-Aug method requires the OLS
estimation of $G$ independent linear models, each having $M$
observations. Note that, since the data matrices are fixed, they need
to be factorised only once.  In absence of breakdowns the method
terminates in at most $m-n = G(M - \overline{k})$ iterations, with
$\overline{k} = \frac1G \sum_{i=1}^G k_i$.
The main computational complexity of the dominating tasks are
reported in Table \ref{tab:SURcompl}.
\begin{table}[h]
  \small
  \centering
  \begin{tabular}{ll@{\,$\times$\,}l}
    \hline \\[-2.3ex]
    \multicolumn{1}{l}{Task} 
     & \multicolumn{1}{l}{Compl. complexity} 
     & Nr. of Iterations
    \\[.3ex]
    \hline \\[-2.3ex]
    QRDs of $\bm{X}_1,\ldots,\bm{X}_G$ 
      & $M^2G \overline{k}$ & $1$ 
    \\[.3ex]
    $\bm\Sigma \bm{u}$  
      & $MG^2$ & $G (M- \overline{k})$ 
    \\[.3ex]
    $\bm{X} \bm{v}$       
      & $MG \overline{k}$ & $G(M- \overline{k})$
    \\[.3ex]
    $\bm{X}^{*T}\bm\xi$ 
      & $G k_{max}^2 + MG \overline{k}$  & $G(M- \overline{k})$
    \\[.3ex]
    $\bm\Pi \bm\xi$  
      & $M G \overline{k}$ & $G(M- \overline{k})$
    \\[.3ex]
    \hline \\[-2.3ex]
    Total complexity
      & \multicolumn{2}{l}{
        $(MG^3 + G^2 k_{max}^2 + M G^2 \overline{k}) (M- \overline{k})
        + M^2 G \overline{k}$
        }
    \\
    \hline
  \end{tabular}
  \medskip

  \caption{%
    Computational cost of the PCG-Aug method for estimating the SUR
    model. 
  }
  \label{tab:SURcompl}
\end{table}

Under specific assumptions on the dimensions, the above expression
further simplifies. For instance,
\begin{itemize}
\item %
  If $M-\overline{k} = O(1)$ and $O(\overline{k}) = O(k_{max}) = O(M)$
  then the computational complexity is of the order
  $MG( G^2 + M G + M^2)$.
\item %
  If $O(M-G) = O(M)$ and $O(\overline{k}) = O(k_{max}) = O(G)$, then
  the complexity becomes $M^2G^4$.
\end{itemize}
When $\bm\Sigma_0$ is non-singular, the GLS estimator for the SUR
model can be computed by means of the QRD of the matrix
$\bm{A} = (\bm{\Sigma}_0^{-\frac12} \otimes \bm{I}_M) (\oplus_i
\bm{X}_i)$, %
where $\bm\Sigma_0^{-\frac12}$ denotes the inverse of a square root of
$\bm\Sigma_0$ (i.e. Cholesky factor). Since that matrix $\bm{A}$ is a
non-sparse matrix having dimensions $GM \times G\overline{k}$, the
cost of that computation is of the order $M^2 G^3 \overline{k}$.

\begin{remark}\label{rem:reduction}
  Let $q$ be the rank of the matrix
  $\bm{W} = (\bm{X}_1 \, \bm{X}_2 \, \cdots \, \bm{X}_G)$, then by
  means of the QRD of $\bm{W}$ it is possible to transform the SUR
  model specified in \eqref{eq:glm} and \eqref{eq:SURdefs} to an
  equivalent SUR model where the number of observation in each
  equation is $q$ \cite{Belsley:92,FoschiBelsleyKontoghiorghes:03}.
  The computational cost of that reduction is $M^2G \overline{k}$.  In
  the worst case the model is reduced to a model where each equation
  has $G\overline{k}$ observations and the number of iterations in
  that case reduces to $G (G-1) \overline{k}$. The whole procedure
  has, then, a computational complexity of the order
  $ M^2 G \overline{k} + G^5 \overline{k}{}^2 + G^3 k^2_{max}
  \overline{k} + G^4 \overline{k}{}^3.  $
\end{remark}

\subsection{Experimental Tests}

Here the performances of the multivariate RGLM method developed in
Section \ref{sec:MVRGLM1} and PCG-Aug method adaption to the SUR model
developed in section \ref{sec:SUR} are compared. The first approach
will be referred to as \-MVRGLM and the latter simply as PCG-Aug.
Additionally, a PCG method applied to the normal equations (PCG-NE)
will be included in the comparison.  In the first experiment these
methods are applied to the SUR estimation of the Fair's macro
econometric model \cite{fair:book1994}. It should remarked that the
point here is not the proposal of an estimator for that model, but
rather to test the performances of the proposed SUR estimation
procedure on some real economic data.  In the following tests, the
data matrices have been preprocessed to reduce the model dimension by
the technique considered in Remark \ref{rem:reduction}.

The performances of the three methods are reported in Figure
\ref{fig:fair1}. Figures \ref{fig:fair1a} and \ref{fig:fair1b} shows
the RMSE against the iteration number, while in Figure
\ref{fig:fair1c} the convergence is expressed as function of the
execution time. Note that, using Matlab as an experimental platform,
due to a very low precision in measuring execution time and to the
platform overheads, these execution times should be taken as lightly
indicative.
Clearly, Figure \ref{fig:fair1} shows the gain obtained by properly
choosing the preconditioner matrix $\bm{D}$ and the superiority of the
augmented system approaches against the usual normal equation
setup. Others preconditioners have been considered and tested for the
PCG-NE method. Since all of them had worse performances their
convergence have not been included in Figure \ref{fig:fair1}.

\begin{figure}[h]
  \centering

  \subfigure[]{\label{fig:fair1a}\includegraphics*[scale=.5,trim=4 0 20 0]{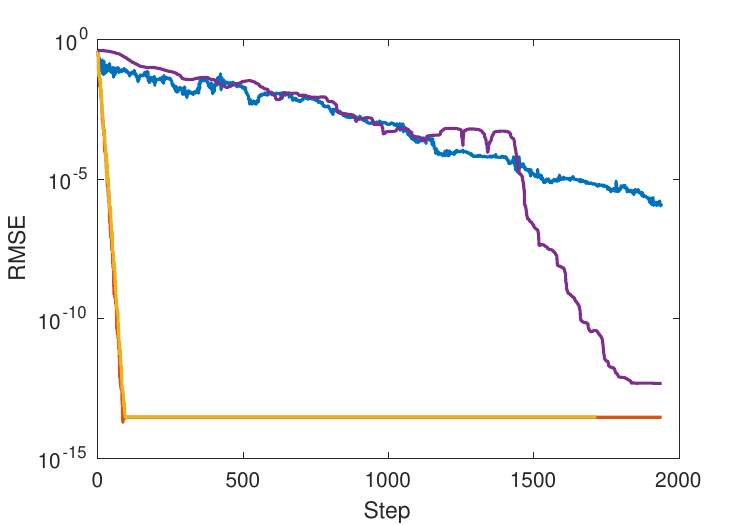}}
  \subfigure[]{\label{fig:fair1b}\includegraphics*[scale=.5,trim=3 0 20 0]{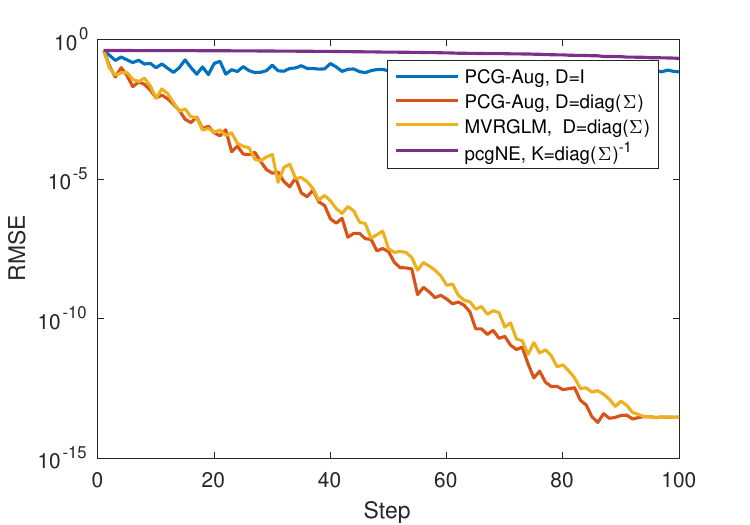}}

  \subfigure[]{\label{fig:fair1c}\includegraphics*[scale=.5,trim=4 0 20 0]{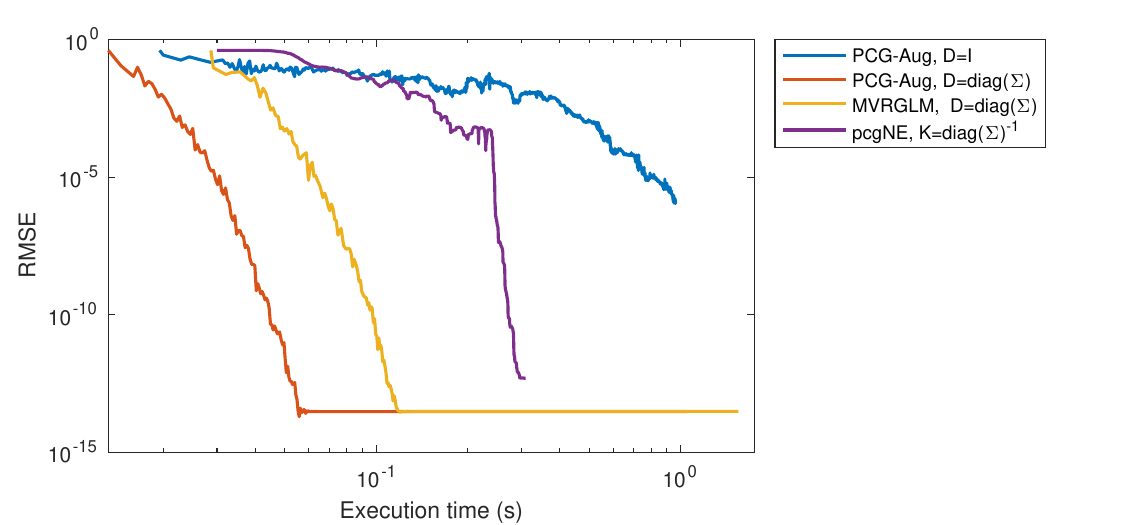}}

  \caption{\small %
    Convergence of the PCG-Aug, PCG-NE and restricted PCG-Aug approaches
    for computing the SUR-GLS estimator of the Fair's model. }
  \label{fig:fair1}

\end{figure}

\bigskip

A second set of tests have been performed for estimating VAR models
with parameter restrictions
\cite{FoschiKontoghiorghes:03b,HamiltonBook,Srivastava:book}. Estimation
of these models reduces to the estimation of a mutivariate RGLM which
can be done by means of MVRGLM, PCG-Aug or PCG-NE methods. Six
different models have been tested. All the models have the same
dimensions $M=300$, $G=12$ and $N=60$ (using the notation of section
\ref{sec:MVRGLM1}).
The rows of $\bm{Z}_0$ follows a VAR(4) model whose largest root is
reported as $\lambda_{max}$ in Table \ref{tab:varStats}.  That Table
reports also the sparsity of the matrix $\bm{B}$ and the condition
number of the different matrices involved in the model.  It should be
noticed that non-stationary models ($\lambda_{max}= 1.05$) have
ill-conditioned regressor matrices $\bm{X}_0$. These six experiments
combines different ill-conditioning on $\bm{X}_0$ and/or $\bm\Omega_0$
with different sparsity levels of $\bm{B}$. Moreover, the condition
number of the normal system can become extremely large, leading to a
stalling PCG-NE's convergence. On the other side, even though the
augmented system matrix can become highly ill-conditioned, that
degeneracy is cured by the trivial choices for the preconditioner.  In
Table \ref{tab:varStats}, $\bm{K}_1$ and $\bm{K}_2$ refer,
respectively, to the choices $\bm{D}_1 = \alpha \bm{I}$ with
$\alpha = \max_i \bm\Sigma_{ii}$ and $\bm{D}_2 = \diag(\bm\Sigma)$.
\begin{table}[bhtp]
  \centering

\footnotesize
\begin{tabular}{cccccHHcccc}
  \hline
  Model 
  & Sparsity 
  & $\lambda_{max}$ 
  &  \multicolumn{8}{c}{Condition number of}
  \\
  nr. & factor && $\bm{X}_0$  
  & $\bm\Omega_0$
  & $\bm\Omega_0^{-1} \otimes \bm{Z}_0^T \bm{Z}_0$ 
  & $\bm{X}$ 
  & $\bm{X}^T\bm\Sigma^{-1}\bm{X}$ 
  & $\bm{G}$
  & $\bm{K_1G}$
  & $\bm{K_2G}$
    \\
  \hline
  1 &    26\% &  0.90 & 8.87e+00 & 4.48e+00 & 3.53e+02 & 6.49e+00 & 5.16e+01 & 9.69e+03 & 4.48e+00 & 5.15e+00  \\
  2 &    26\% &  1.05 & 1.64e+02 & 4.48e+00 & 1.21e+05 & 1.13e+02 & 1.41e+04 & 1.78e+05 & 4.48e+00 & 5.28e+00  \\
  3 &    26\% &  0.90 & 1.39e+01 & 4.48e+02 & 8.64e+04 & 6.87e+00 & 3.25e+03 & 8.75e+05 & 4.48e+02 & 5.11e+02  \\
  4 &    26\% &  1.05 & 2.07e+02 & 4.48e+02 & 1.93e+07 & 1.14e+02 & 6.81e+05 & 1.59e+07 & 4.48e+02 & 5.19e+02  \\ 
  5 &    80\% &  0.90 & 1.33e+01 & 4.48e+02 & 7.95e+04 & 1.22e+01 & 4.50e+04 & 4.33e+04 & 1.81e+01 & 1.64e+01  \\
  6 &    80\% &  1.05 & 2.54e+04 & 4.48e+02 & 2.90e+11 & 2.33e+04 & 1.39e+11 & 1.20e+08 & 2.21e+01 & 2.02e+01  \\
  \hline
\end{tabular}
\medskip

\caption{ %
  Statistics for the six VAR models with parameter
  restrictions. 
}
  \label{tab:varStats}
\end{table}

Figures \ref{fig:var} and \ref{fig:varCont} show the convergence of
the PCG-NE, of the PCG-Aug (with preconditioners $\bm{K}_1$ and
$\bm{K}_2$) and of the MVRGLM (with preconditioners $\bm{K}_1$ and
$\bm{K}_2$) methods. In these figure the norm of the normal equation
residuals $\bm{X}^T\bm\Sigma^{-1}( \bm{X}\hat{\bm\beta} - \bm{y})$ is
plotted against the iteration number or the execution time.
These experiments show that, the PCG-NE method has good performances
only for well conditioned problems. On the contrary, PCG-Aug and
MVRGLM methods are more robust and do not have a remarkable loss of
performances when either the regressor or covariance matrices becomes
ill-conditioned. Note that, even though the proposed methods are
implemented in Matlab, the execution time required to compute the
exact solution is comparable if not much better than that of the
Matlab solver.

These experimental results also show that, for this application, there
is not any gain in choosing a diagonal matrix $\bm{D}$ over
a properly scaled identity matrix.

\begin{figure}[p]
  \centering
  \small

  \subfigure[VAR model 1]{
    \includegraphics*[scale=.4,trim=4 0 20 0]{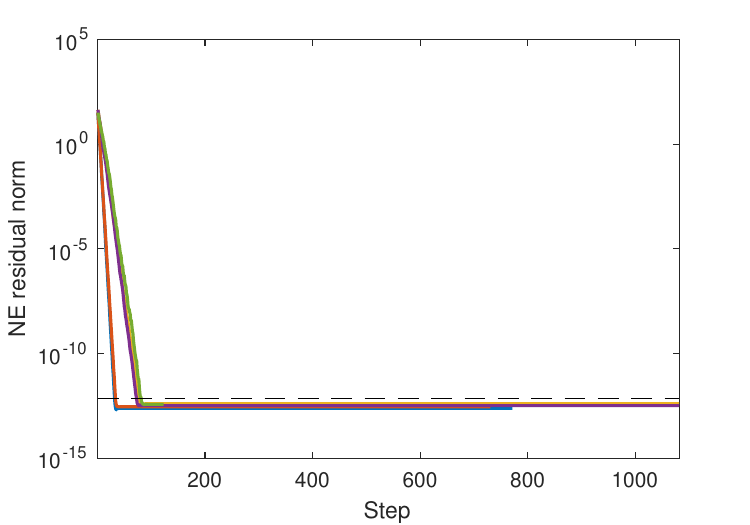}
    \quad
    \includegraphics*[scale=.4,trim=4 0 20 0]{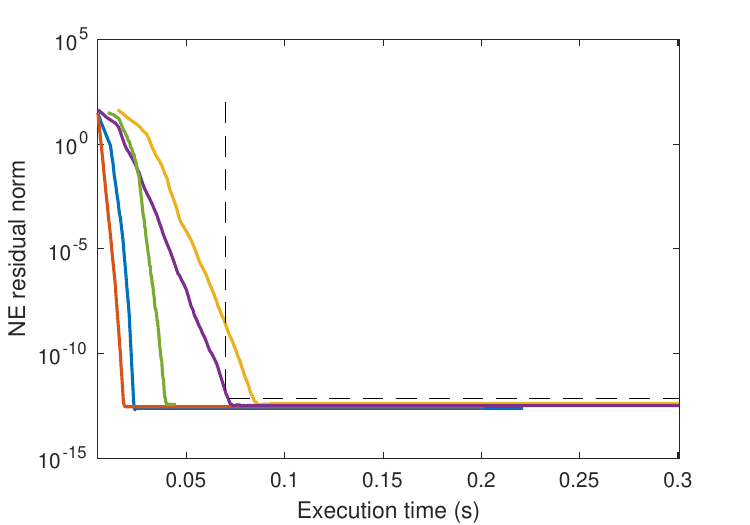}
  }
  
  \subfigure[VAR model 2]{
    \includegraphics*[scale=.4,trim=4 0 20 0]{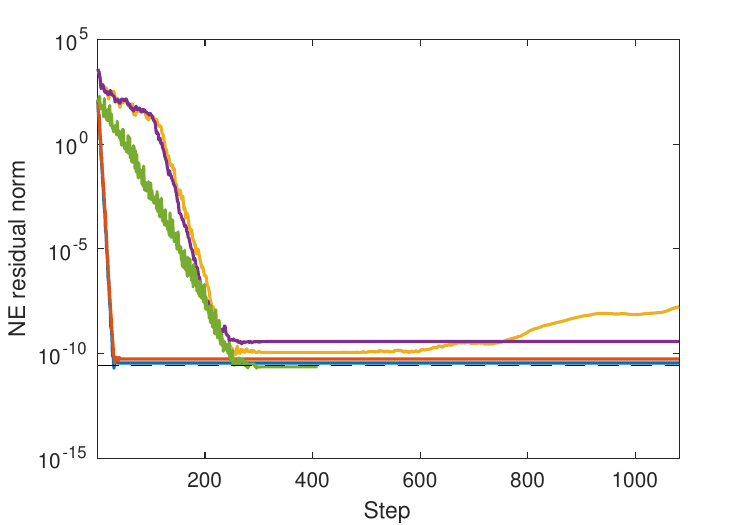}
    \quad
    \includegraphics*[scale=.4,trim=4 0 20 0]{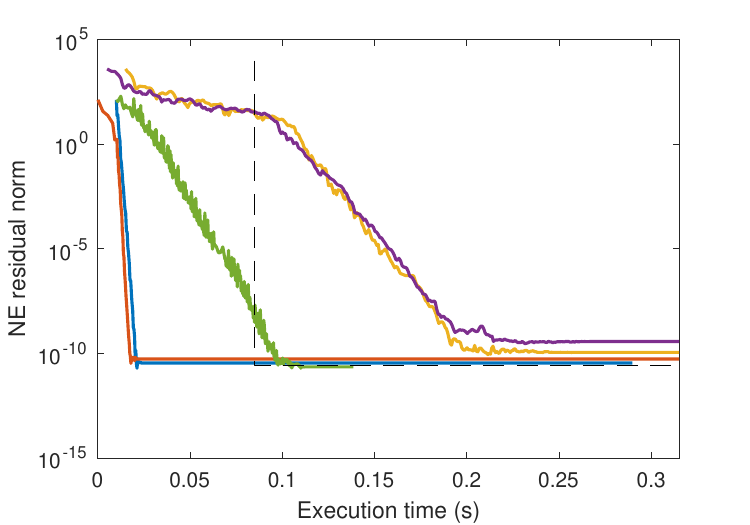}
  }

  \subfigure[VAR model 3]{
    \includegraphics*[scale=.4,trim=4 0 20 0]{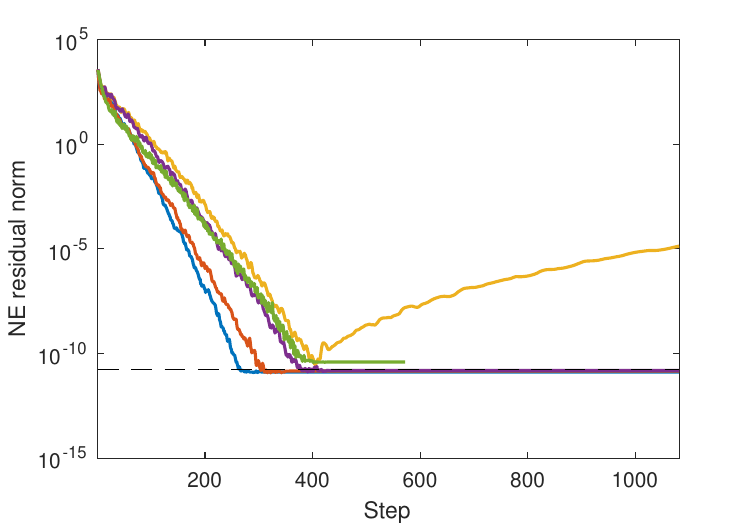}
    \quad
    \includegraphics*[scale=.4,trim=4 0 20 0]{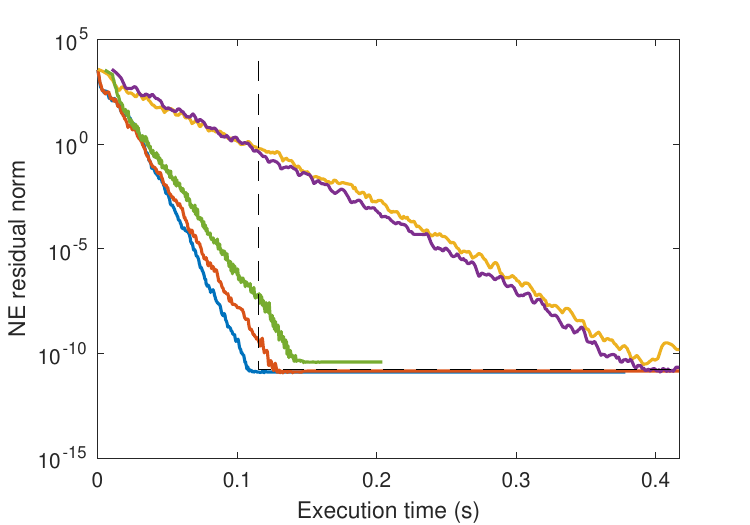}
  }

  \caption{\small 
    Convergence of the PCG-NE (green), PCG-Aug method with
    $\bm{D} = \alpha\bm{I}$ and $\bm{D} = \diag(\bm{\Sigma})$ (blue and red)
    and of the MVRGLM method with $\bm{D} = \alpha\bm{I}$ and
    $\bm{D} = \diag(\bm{\Sigma})$ (yellow and purple) for estimating
    the restricted VAR models 1-4.  
    The execution time and precision of Matlab solver for the
    augmented system is shown as dashed black vertical and horizontal
    lines.  }
  \label{fig:var}
\end{figure}

\begin{figure}[p]
  \centering
  \small

  \subfigure[VAR model 4]{
    \includegraphics*[scale=.4,trim=4 0 20 0]{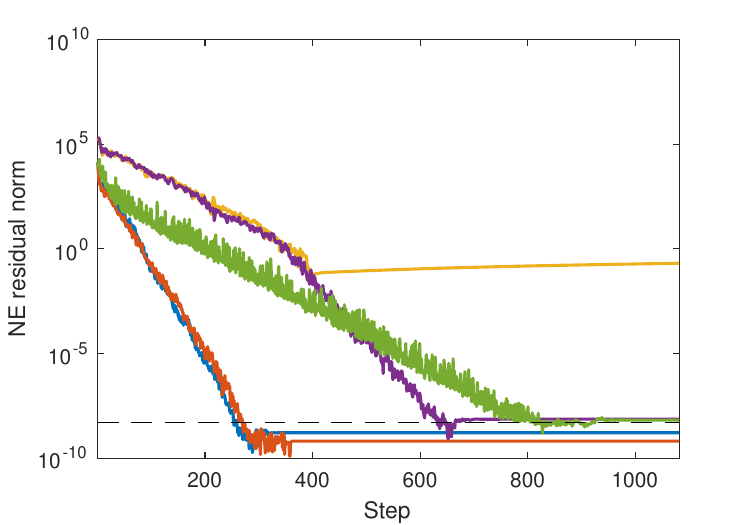}
    \quad
    \includegraphics*[scale=.4,trim=4 0 20 0]{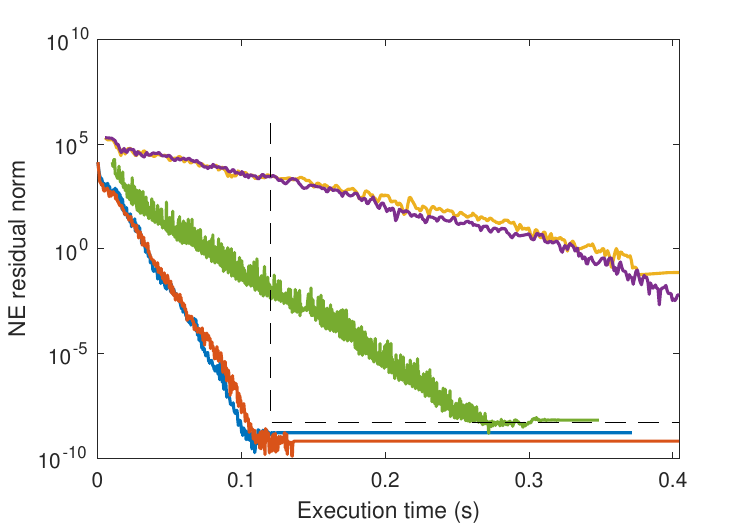}
  }

  \subfigure[VAR model 5]{
    \includegraphics*[scale=.4,trim=4 0 20 0]{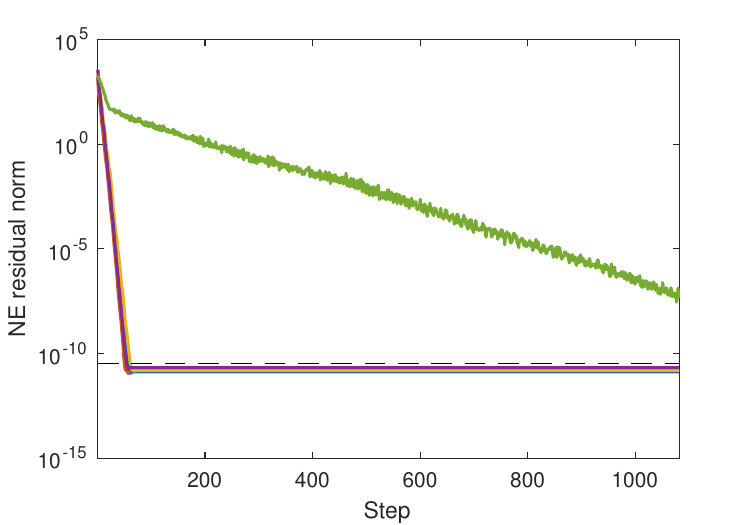}
    \quad
    \includegraphics*[scale=.4,trim=4 0 20 0]{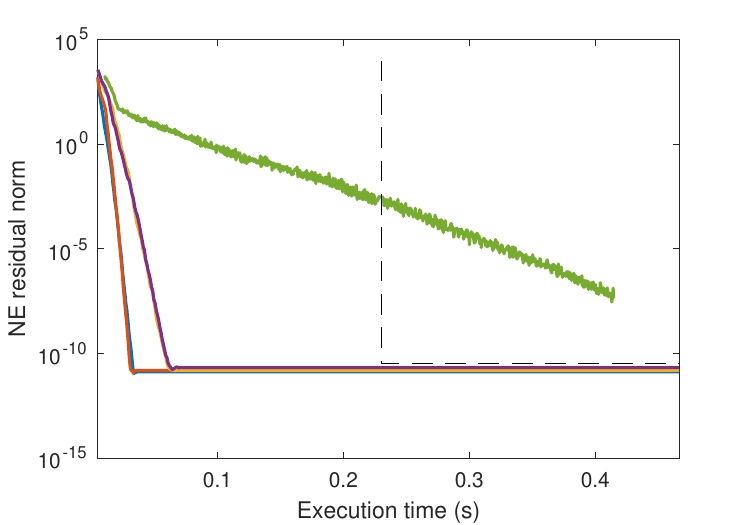}
  }

  \subfigure[VAR model 6]{
    \includegraphics*[scale=.4,trim=4 0 20 0]{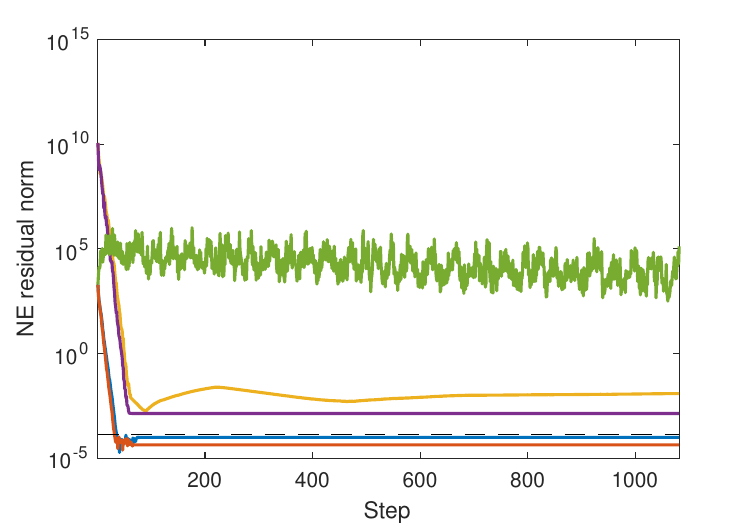}
    \quad
    \includegraphics*[scale=.4,trim=4 0 20 0]{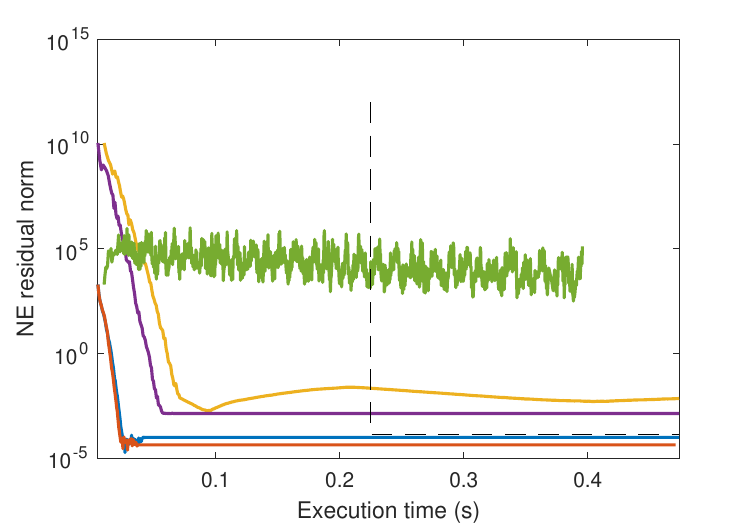}
  }

  \caption{\small 
    (Cont. from Figure \ref{fig:var}).
    Convergence of the PCG-NE, PCG-Aug and MVRGLM methods for estimating
    the restricted VAR models 5-6.  
  }
  \label{fig:varCont}
\end{figure}

\section{Conclusions}

The solution by means of the PCGs method of the augmented system
formulation for the GLS estimator has been considered.  The indefinite
preconditioner, originally proposed in \cite{LuksanVlcek:98} for
solving constrained quadratic programming problems, is used. The
resulting method, uses OLS estimations to iteratively updates an
estimator which, in exact precision, after a finite number of
iteration will provide the GLS estimator. The method is particularly
advantageous when that OLS estimation can be computed in efficient
way.

Moreover, contrary to normal equation based estimators, this approach
does not require a non-singular covariance matrix. The requirements
are those of a well specified GLM: the covariance matrix need to be
positive definite on the null space of the regressor matrix.

Some inferential properties of this methods are considered.  In
particular, contrary to what happen for the PCG method applied to the
normal equations (PCG-NE), the intermediate iterates of the PCG-Aug
method provide unbiased estimators for the parameters.  Both a
mathematical proof and an Monte Carlo experiment to test this property
are provided.  In the simulations the PCG-Aug method showed also
better performances both on average and on the 5\% and 1\% worst
cases.
Notice that, modulo some normalisation, the PCG-NE method is nothing
else than what the statistical data analysis literature call Partial
Least Squares regression method \cite{Elden:2004}.

The PCG-Aug method have been applied to some statistical problems that
can be written as specific structured GLMs. More precisely univariate
and multivariate GLMs with linear restrictions on the parameters that
can be written as a GLMs with singular covariance matrices.  The
Seemingly Unrelated Regressions model have been also considered. That
model can be written as a GLM where the regressor matrix and the
covariance matrices have, respectively, block diagonal and Kronecker
product structures. In the latter two models, OLS estimation is
computationally much cheaper than GLS estimation
\cite{Srivastava:book}. This allows the PCG-Aug methods to have very
good performances.
Numerical experiments have been performed to confirm this claim.  For
those problems the PCG-Aug methods have shown to be faster than and as
numerical precise as direct methods (the Matlab solver).
It should be remarked that the PCG-Aug approach combines the
advantages of direct methods with those of iterative methods. In the
OLS step the structure of the model is exploited as much as possible,
while between the iterations the un-exploitable structure is taken
into account.
Notice that here, contrary to what is usually found in the numerical
linear algebra literature, the use of a PCG method is motivated by the
structure of the problem and not by its sparsity. For instance, the
diagonal blocks in the SUR and VAR models are full matrices.

The same approach can be applied to different structured linear
statistical problems. For instance, in panel data the covariance
matrix is a small rank update of an identity or of a diagonal
matrix. In those applications adding more structure to the models,
like introducing autoregressive dynamics or heteroskedasticity, does
not allow for computationally efficient factorisation methods
\cite{Baltagi:book}. Other possible applications can be found in
spatial data analysis where the regressor matrix and or the covariance
matrix have often a block diagonal, a Kronecker product or some other
sparse structure.

\bibliographystyle{amsplain}
\providecommand{\bysame}{\leavevmode\hbox to3em{\hrulefill}\thinspace}
\providecommand{\MR}{\relax\ifhmode\unskip\space\fi MR }
\providecommand{\MRhref}[2]{%
  \href{http://www.ams.org/mathscinet-getitem?mr=#1}{#2}
}
\providecommand{\href}[2]{#2}


\clearpage
\appendix
\section{Proofs}\label{sec:proofs}

\begin{proof}[Proof of Lemma \ref{thm:AugBLUE}]
  By applying the orthogonal transformation $\bm{Q}$ to the first block of
  rows and columns of $\bm{G}$, the augmented system \eqref{eq:aug1} can be
  rewritten in the following form
  \begin{align}\label{eq:rotatedAug}
    \pmx{
      \bm\Sigma_{RR} & \bm\Sigma_{RN} & \bm{R} \\
      \bm\Sigma_{NR} & \bm\Sigma_{NN} & \bm{0} \\
      \bm{R}^T & \bm{0} & \bm{0}
    }
    \pmx{ \bm{w}_R \\ \bm{w}_N \\ \bm{b}_{Aug} }
    =
    \pmx{ \bm{y}_R \\ \bm{y}_N \\ \bm{0} }
  \end{align}
  where
  $\bm\Sigma_{ij} = \bm{Q}_i^T \bm\Sigma \bm{Q}_j$,
  $\bm{w}_{i} = \bm{Q}_i^T \bm{w}$
  and $\bm{y}_i = \bm{Q}_i^T \bm{y}$,
  for $i,j \in \{R,N\}$.
  As $\bm\Sigma_{NN}$ is positive definite, the solution to
  \eqref{eq:rotatedAug} is given by
  \begin{align*}
    \bm{w}_R &= \bm{0},
    &
    \bm{w}_N &= \bm\Sigma_{NN}^{-1} \bm{y}_N
    &\text{and}&&
    \bm{b}_{Aug} &= \bm{R}^{-1}( \bm{y}_R - \bm{\Sigma}_{RN}\bm\Sigma_{NN}^{-1}\bm{y}_N).
  \end{align*}
  It follows that the solution to \eqref{eq:aug1} is given by \eqref{eq:augsol}.

  In order to show that $\bm{b}_{Aug}$ is the BLUE for $\bm\beta$,
  note that $\bm{P}_N \bm{X}= \bm{X}$. Then, the estimator
  $\bm{b}_{Aug}$ can be rewritten as
  $\bm{b}_{Aug} = \bm{b} + \bm{R}^{-1}\bm{Q}_R^T \bm{P}_N \bm\eps$ and
  so $\bm{b}_{Aug}$ is unbiased and its covariance matrix is given by
  \begin{align*}
    \cov(\bm{b}_{Aug}) = \bm{R}^{-1}\bm{Q}_R^T\bm{P}_N \bm\Sigma \bm{Q}_R \bm{R}^{-T},
  \end{align*}
  where the property $\bm{P}_N\bm{\Sigma}\bm{P}_N = \bm{P}_N\bm{\Sigma}$ has been used.
  Next, to prove the optimality of this estimator, consider an
  alternative linear unbiased estimator
  $\tilde{\bm{b}} = \bm{A}^T \bm{y}$. Being $\tilde{\bm{b}}$ unbiased,
  it is necessary that $\bm{A}^T\bm{X}=\bm{I}$.  This is equivalent to
  require that $A$ is given by
  $\bm{A} = \bm{Q}_R\bm{R}^{-T} + \bm{Q}_N \bm{A}_N$ for some
  $\bm{A}_N \in \Real^{(m-n) \times n}$. Next, since the covariance of
  $\tilde{\bm{b}}$ is given by
  \begin{align*}
    \cov(\tilde{\bm{b}}) = \cov(\tilde{\bm{b}}-\bm{b}_{Aug}) + \cov(\bm{b}_{Aug})
    + \cov(\bm{b}_{Aug}, \tilde{\bm{b}}-\bm{b}_{Aug})
    + \cov(\tilde{\bm{b}}-\bm{b}_{Aug}, \bm{b}_{Aug}).
  \end{align*}
  and  $\tilde{\bm{b}} - \bm{b}_{Aug} = \bm{A}^T \bm{\eps} -
  \bm{R}^{-1}\bm{Q}_R^T \bm{P}_N \bm{\eps}$, 
  then
  \begin{align*}
    \cov(\bm{b}_{Aug}, \tilde{\bm{b}}-\bm{b}_{Aug})
    &=
    \bm{R}^{-1} \bm{Q}_R^T \bm{P}_N \bm{\Sigma} (\bm{A} - \bm{P}_N\bm{Q}_R\bm{R}^{-T})
    \\
    &=
    \bm{R}^{-1} \bm{Q}_R^T( 
    \bm{P}_N\bm\Sigma\bm{Q}_R\bm{R}^{-T} 
    + \bm{P}_N\bm{\Sigma}\bm{Q}_N\bm{A}_N 
    - \bm{P}_N\bm\Sigma\bm{Q}_R \bm{R}^{-T} 
    )
    \\
    &=
    \bm{R}^{-1} \bm{Q}_R^T \bm{P}_N \bm\Sigma \bm{Q}_N \bm{A}_N
    = 0,
  \end{align*}
  where it has been used the property $\bm{P}_N \bm\Sigma \bm{Q}_N = \bm{0}$. 
  This proves that $\cov(\tilde{\bm{b}}) - \cov(\bm{b}_{Aug})$.
\end{proof}

\end{document}